\documentclass[11pt]{amsart}
\usepackage[T1]{fontenc}
\usepackage{multirow}
\usepackage{amsmath}
\usepackage{amssymb}
\usepackage{cite,enumerate}
\usepackage{setspace}
\usepackage{xcolor}
\definecolor{proofgray}{gray}{0.45}

\usepackage{amsthm}

\usepackage{graphicx}

\usepackage[colorlinks,
            linkcolor=blue,
            anchorcolor=blue,
            citecolor=blue
            ]{hyperref}

\newcommand*{\Res}{\operatorname{Res}}
\newcommand*{\res}{\operatorname{res}}

\newcommand*{\Id}{\mathrm{Id}}

\newcommand{\dvg}{\operatorname{div}}

\newcommand{\Hess}{\operatorname{Hess}}
\newcommand{\Vol}{\operatorname{Vol}}
\newcommand{\Area}{\operatorname{Area}}

\newcommand{\ord}{\operatorname{ord}}
\newcommand{\Ogroup}[1]{\mathrm O({#1})}
\newcommand{\Proj}{\operatorname{Proj}}
\newcommand{\Ric}{\operatorname{Ric}}
\newcommand{\sn}{\operatorname{sn}}
\newcommand{\cn}{\operatorname{cn}}
\newcommand{\tn}{\operatorname{tn}}
\newcommand{\Rm}{\operatorname{Rm}}
\newcommand{\Spec}{\operatorname{Spec}}
\newcommand{\Span}{\operatorname{Span}}
\newcommand{\Tr}{\operatorname{Tr}}

\renewcommand{\Re}{\operatorname{Re}}

\newtheorem{theorem}{Theorem}[section]
\newtheorem{lemma}[theorem]{Lemma}
\newtheorem{proposition}[theorem]{Proposition}
\newtheorem{corollary}[theorem]{Corollary}
\theoremstyle{definition}

\newtheorem{remark}{Remark}[section]

\numberwithin{equation}{section}

\title[Steklov Spectral Uniqueness via Guillemin-Wodzicki Residues]{Steklov Spectral Uniqueness of Geodesic Balls in 3-Dimensional Space Forms via Guillemin--Wodzicki Residues}

\author{Zuoqin Wang, Hanzhang Yun}
\thanks{Partially supported by NNSFC No. 12571064.}
\address{School of Mathematical Sciences\\
	University of Science and Technology of China\\
	Hefei\\ 230026\\ P.R. China}
\email{wangzuoq@ustc.edu.cn}

\address{Department of Mathematics\\
	Washington University in St. Louis\\  1 Brookings Drive \\
	St. Louis \\ MO   63130 \\ U.S.A.}
\email{hanzhang@wustl.edu}

\begin{document}

	\begin{abstract}		
		Let $\Lambda$ be the Dirichlet-to-Neumann operator on the boundary of a compact three-dimensional Riemannian manifold.  Using the Lee--Uhlmann full-symbol formula and Weyl's invariant theory, we compute  the Guillemin--Wodzicki residues of $\Lambda$ and $\Lambda^2$ explicitly, and hence the first two logarithmic coefficients in the Steklov heat trace. As an application, we prove that geodesic balls in simply connected three-dimensional space
		forms are determined by their Steklov spectra among smooth domains in the same space form. More generally, we obtain a rigidity theorem in the class of compact constant-curvature manifolds with smooth, not necessarily connected, boundary.  In the Euclidean case, we also prove spectral uniqueness for concentric spherical shell regions.
	\end{abstract}

	\maketitle

	\section{Introduction}
	
	Let $(\Omega,g)$ be a compact connected Riemannian manifold of dimension $n$ with nonempty smooth boundary $\Sigma$, and let $\Delta=\Delta_g$ denote the Laplace--Beltrami operator on $\Omega$. Given any smooth function $u\in C^\infty(\Sigma)$, the boundary value problem
	\[\left\{\begin{aligned}
		& \Delta f = 0 && \text{in $\Omega$},\\
		& f = u && \text{ on $\Sigma$}\\
	\end{aligned}\right.\]
	has a unique solution $f=\mathcal Hu\in C^\infty(\Omega)$. The {\em Dirichlet-to-Neumann operator on $\Sigma$} is the operator  $\Lambda(\Omega,g): C^\infty(\Sigma)\to C^\infty(\Sigma)$ defined by  
	\[\Lambda(\Omega,g) u =  {\partial_\nu(\mathcal Hu)} \in C^\infty(\Sigma),\]
	where $\nu$ is the  outward unit normal vector field along $\Sigma$.  Equivalently, the Dirichlet-to-Neumann eigenvalue problem 
	\[
	\Lambda(\Omega,g)u=\mu u
	\]
	is the boundary formulation of the {\em Steklov eigenvalue problem}
	\[
	\begin{cases}
		\Delta_g f=0 & \text{in }\Omega,\\
		\partial_\nu f=\mu f & \text{on }\Sigma.
	\end{cases}
	\]
	Thus the Dirichlet-to-Neumann spectrum and the Steklov spectrum coincide.  Since $\Lambda(\Omega,g)$ is a nonnegative, self-adjoint, elliptic classical pseudodifferential operator of order one on $\Sigma$ \cite{taylorPartialDifferentialEquations2011}, its spectrum is discrete and can be written, with multiplicities, as
	\[
	\Spec(\Lambda(\Omega,g)):
	\qquad
	0=\mu_0(\Omega,g)<\mu_1(\Omega,g)\leq \mu_2(\Omega,g)\leq\cdots\nearrow+\infty.
	\]
	The operator therefore has two complementary geometric interpretations: it is produced by an elliptic boundary value problem in the interior, but it acts as a pseudodifferential operator entirely on the boundary.  This interaction between interior and boundary geometry lies at the heart of the Steklov inverse spectral problem and is one of the main reasons why the problem is both subtle and rich.
	
	\subsection{Inverse spectral results for the Steklov problem}
	
	The inverse spectral problem asks which geometric features of $(\Omega,g)$ are determined by $\Spec(\Lambda(\Omega,g))$.  The leading asymptotics already recover basic boundary data.  Indeed, Weyl's law (cf. \cite[Chap. 29]{hormanderAnalysisLinearPartial2009}) gives
	\begin{equation}
		\label{eq: Weyl law}
		\#\{j:\mu_j\leq \mu\}
		=C_{n-1}\Vol(\Sigma,g)\mu^{n-1}+O(\mu^{n-2}),
		\qquad \mu\to+\infty,
	\end{equation}
	where $C_{n-1}>0$ is universal.  Consequently, the spectrum determines the dimension of $\Sigma$ and its $(n-1)$-dimensional volume.  In dimension two this information already leads to several inverse spectral results.  For a simply connected planar domain, Weinstock's inequality \cite{weinstockInequalitiesClassicalEigenvalue1954} states that
	\[
	\mu_1(\Omega)\mathrm{Length}(\partial\Omega)
	\leq
	\mu_1(\mathbb D)\mathrm{Length}(\partial\mathbb D)
	=2\pi,
	\]
	with equality if and only if $\Omega$ is a disk.  Since the boundary length is audible by \eqref{eq: Weyl law}, the disk is determined by the Steklov spectrum among simply connected planar domains. This result was first obtained by Edward \cite{edwardInverseSpectralResult1993} via the value of the spectral zeta function associated with the DtN spectrum. For Riemannian surfaces, more refined asymptotic results together with a number-theoretic argument in \cite{girouardSteklovSpectrumSurfaces2014} show that the one-dimensional boundary is exceptionally rigid: the Steklov spectrum detects the number of boundary components and their lengths.  In particular, disks are spectrally determined even without imposing the boundary connectedness assumption.
	Very recently, using the spectral zeta function and the zeta-regularized determinant, it was proved in \cite{jinSteklovSpectralGeometry2026} that  each annulus bounded by two concentric circles is uniquely determined by its Steklov spectrum among all planar domains with smooth boundary,  providing the first example of a non-simply connected Euclidean domain that is uniquely determined by its Steklov spectrum.

	The higher-dimensional problem is substantially more difficult.  Weinstock's argument is based on the Riemann mapping theorem, and its higher-dimensional analogues require additional geometric hypotheses \cite{bucurWeinstockInequalityHigher2021}.  Polterovich and Sher \cite{polterovichHeatInvariantsSteklov2015} proved the first general rigidity theorem for a three-dimensional ball: a Euclidean ball is determined by its Steklov spectrum among smooth bounded domains in $\mathbb R^3$ with connected boundary.  Their proof uses the first local coefficients in the heat trace expansion, together with additional information encoded by the multiplicities of the ball spectrum.  Very recently, Speciel \cite{specielSteklovRigidityEuclidean2026} removed the connectedness assumption in the Euclidean setting, and in fact proved a stronger asymptotic rigidity result in every dimension $n\geq 3$. Thus, Euclidean balls are determined by their Steklov spectrum among all smooth bounded Euclidean domains, possibly with disconnected boundary. 
	
	Our first main theorem extends the rigidity phenomenon from Euclidean space to space forms of nonzero curvature in dimension three. Let $\mathbb M_\kappa$ denote the simply connected three-dimensional space form of constant sectional curvature $\kappa$, and let $\mathbb B_\kappa(\rho)\subset\mathbb M_\kappa$ be the
	geodesic ball of radius $\rho>0$. 
	
	\begin{theorem}
		\label{thm: determine 3d geodesic balls}
		Let $\kappa\in\mathbb R$, and let $(\Omega,g)$ be a compact connected three-dimensional Riemannian manifold of constant sectional curvature $\kappa$ with smooth boundary $\Sigma$.  Suppose that
		\[
		\Spec(\Lambda(\Omega,g)) = \Spec(\Lambda(\mathbb B_\kappa(\rho)))
		\]
		for some $\rho>0$.  In the case $\kappa > 0$ we assume  in addition that $2\sqrt{\kappa}\rho\le\pi$. 
		Then $(\Omega,g)$ is isometric to $\mathbb B_\kappa(\rho)$.  
	\end{theorem}
	
	No connectedness assumption is imposed on $\Sigma$. When $\kappa=0$ and $\Omega$ is a Euclidean domain, the corresponding rigidity statement is also contained in the recent result of Speciel \cite{specielSteklovRigidityEuclidean2026}. Theorem \ref{thm: determine 3d geodesic balls} extends this rigidity phenomenon to three-dimensional space forms of arbitrary constant sectional curvature and, more generally, is formulated for compact constant-curvature manifolds with boundary rather than only for domains in the simply connected model space.

	For domains in a simply connected space form, we are able to remove this radius assumption in the case $\kappa>0$, and thus prove the following consequence. 
	
	\begin{corollary}
		\label{cor: determine 3d geodesic balls in orientable space forms}
		Let $\kappa\in\mathbb R$, and let $\Omega\subset \mathbb M_\kappa$ be a bounded domain with smooth boundary.  If
		\[
		\Spec(\Lambda(\Omega)) = \Spec(\Lambda(\mathbb B_\kappa(\rho))),
		\]
		then $\Omega$ is isometric to $\mathbb B_\kappa(\rho)$.
	\end{corollary}
	
	We also study the Steklov inverse spectral problem for  concentric spherical shell regions. Fix radii $0<r<R$ and denote the concentric spherical shell by
	\[
	A_{r,R}:=B_R(0)\setminus \overline{B_r(0)}.
	\]
	
	\begin{theorem}\label{thm:concentric_spherical_shell_inverse}
		Let $\Omega\subset \mathbb R^3$ be a bounded domain with smooth boundary.  If
		 \[
		 \Spec(\Lambda(\Omega)) = \Spec(\Lambda(A_{r,R})),
		 \]
		 then $\Omega$ is isometric to $A_{r,R}$.
	\end{theorem}

	\subsection{New Steklov spectral invariants}
	
	As in \cite{polterovichHeatInvariantsSteklov2015}, we shall use heat invariants for the Dirichlet-to-Neumann operator to prove the inverse results alluded to above. The key new ingredient is the use of logarithmic heat invariants, which can be computed via the Guillemin--Wodzicki noncommutative residues  $\Res(\Lambda(\Omega,g))$ and  $\Res(\Lambda(\Omega,g)^2)$. We first state our results and postpone further explanations, especially the relation between the heat invariants and the Guillemin--Wodzicki noncommutative residues, to Section~\ref{sec: preliminaries}.
	
	We write $R$ and $R^\top$ for the scalar curvatures of $\Omega$ and $\Sigma$, respectively, and $h$ for the second fundamental form.  For tangent vectors
	$X,Y\in T\Sigma$, set
	\[
	\Rm^\perp(X,Y):=\Rm(X,\nu,Y,\nu), \qquad
	\Ric^\perp:=\Ric(\nu,\nu).
	\]
	The operators $\nabla^\top$, $\operatorname{Hess}^\top$, $\operatorname{div}^\top$, and $\Delta^\top$ are computed with respect to the induced metric on $\Sigma$, with $\Delta^\top$ taken to be negatively definite.  Normal derivatives are taken in boundary normal coordinates. Moreover, we write   
	\[ h^2(X,Y)=\sum_\alpha h(e_\alpha,X)h(e_\alpha,Y),\] where $\{e_\alpha\}$ is any local orthonormal frame  on $\Sigma$. We will prove 
	
	\begin{theorem}
		\label{thm: main thm b1 b2}
		Let $(\Omega,g)$ be a compact three-dimensional Riemannian manifold with smooth boundary $\Sigma$. For simplicity we write $\Lambda=\Lambda(\Omega, g)$. Then
		\[
		\Res(\Lambda(\Omega,g)^k)=(-1)^{k+1}\int_\Sigma \widehat b_k\,|d\sigma_g|,
		\qquad k=1,2,
		\]
		where
		\begin{align*}
			\widehat b_1
			= &\, \frac{\pi}{8} H_1^3 + \frac{3\pi}{16}H_1R^\top - \frac{\pi}{2}H_1H_2 + \frac{\pi}{16}\Delta^\top H_1 \\
			&\, + \frac{\pi}{8}\langle h, \Rm^\bot\rangle + \frac{\pi}{16}H_1\Ric^\bot + \frac{\pi}{16}\partial_{x^n}\Ric^\bot - \frac{3\pi}{32}\partial_{x^n}R, 
		\end{align*}
		and
		\begin{align*}
			\begin{aligned}
				\widehat b_2 = &\ -\frac{\pi}{4} H_1^4 - \frac{3\pi}{8} H_1^2 H_2 + \frac{3\pi}{8} H_1^2 R^\top + \frac{3\pi}{16} H_1^2\Ric^\bot + \frac{3\pi}{16} H_2 R^\top \\
				&\ - \frac{3\pi}{32}\Ric^\bot R^\top - \frac{\pi}{2} H_2^2 + \frac{5\pi}{16} H_2\Ric^\bot - \frac{\pi}{32}(\Ric^\bot)^2 + \frac{3\pi}{8} H_1\langle h,\Rm^\bot\rangle\\
				&\ + \frac{9\pi}{16} H_1\partial_n\Ric^\bot - \frac{9\pi}{16} H_1\partial_n R - \frac{\pi}{8}\langle h,\Hess^\top H_1\rangle - \frac{\pi}{8}|\nabla^\top H_1|^2\\
				&\ - \frac{\pi}{8}\langle\dvg^\top h,\nabla^\top H_1\rangle + \frac{\pi}{16}\Delta^\top H_2 - \frac{\pi}{32}\Delta^\top\Ric^\bot - \frac{\pi}{8}\langle h^2,\Rm^\bot\rangle\\
				&\ - \frac{\pi}{8}\langle\Rm^\bot,\Rm^\bot\rangle - \frac{\pi}{8}\langle h,\Rm^\bot_{;n}\rangle - \frac{\pi}{16}\partial_n^2\Ric^\bot + \frac{3\pi}{32}\partial_n^2 R.
			\end{aligned}
		\end{align*} 
	\end{theorem}

	For the inverse spectral application, the general expressions simplify dramatically in constant sectional curvature.
	
	\begin{corollary}
		\label{cor: b_1 b_2 constant curvature}
		Let $(\Omega,g)$ be a compact three-dimensional Riemannian manifold of constant sectional curvature $\kappa$ with smooth boundary $\Sigma$.  Then
		\begin{equation}
			\label{eq: b1}
			\Res(\Lambda) = \frac{\pi}{8}\int_\Sigma H_1(H_1^2-H_2)\,|d\sigma_g|
			+\frac{3\pi\kappa}{4}\int_\Sigma H_1\,|d\sigma_g|,
		\end{equation}
		and
		\begin{equation}
			\label{eq: b2}
			\begin{aligned}
				\Res(\Lambda^2)={}&
				\frac{\pi}{8}\int_\Sigma
				\left((2H_1^2-H_2)(H_1^2-H_2) + |\nabla^\top H_1|^2\right)\,|d\sigma_g|\\
				&
				-\frac{11\pi\kappa}{8}\int_\Sigma H_1^2\,|d\sigma_g|
				-\frac{7\pi^2\kappa}{4}\chi(\Sigma)
				+\frac{13\pi}{8}\kappa^2\Area(\Sigma,g).
			\end{aligned}
		\end{equation}
	\end{corollary}

	In the flat case, Theorem \ref{thm: main thm b1 b2} also reveals a direct connection with Willmore geometry: the local residue density of $\Lambda$ vanishes precisely when
	\[
	\Delta^\top H_1+H_1(2H_1^2-R^\top)=0,
	\]
	the Euler--Lagrange equation for the Willmore functional
	\[
	\mathcal W(\Sigma)=\int_\Sigma H_1^2\,|d\sigma_g|;
	\]
	see \cite{riviere2008analysis}.  Thus the first logarithmic heat coefficient detects a classical conformal variational equation on the boundary.
	
	\subsection{Arrangement of the paper}
	
	The paper is organized as follows. Section~\ref{sec: preliminaries} recalls basic facts about heat invariants and the Guillemin--Wodzicki residue, with emphasis on the relation between the noncommutative residue and the logarithmic heat coefficients. Section~\ref{sec:metricpolyn} develops the algebraic framework for local scalar invariants near the boundary, and Section~\ref{sec: full symbol}  computes the full symbol of $\Lambda$ and $\Lambda^2$ using the Lee--Uhlmann formula. In Section~\ref{sec:GWResidue}, these ingredients are used to express the Guillemin--Wodzicki residue density of $\Lambda$ as a linear combination of eight invariant polynomials. Their coefficients are determined in Section~\ref{sec: proof of thm} by computations on several model geometries, yielding Theorem~\ref{thm: main thm b1 b2} and Corollary~\ref{cor: b_1 b_2 constant curvature}. Section~\ref{sec:apptoSpecGeom} applies the resulting spectral invariants to detect boundary geometry and topology and proves Theorem~\ref{thm:concentric_spherical_shell_inverse}. Finally, Section~\ref{sec:ProofThm1} specializes to constant-curvature manifolds and proves Theorem~\ref{thm: determine 3d geodesic balls} and Corollary~\ref{cor: determine 3d geodesic balls in orientable space forms}.

	{\bf Notation:} Throughout this paper, we write 
	\begin{itemize}
		\item Greek letter indices $\alpha,\beta,\dots$ range from $1$ to $n-1$;
		\item Latin letter indices $a,b,\dots$ range from $1$ to $n$;
		\item $J=(j_1, \cdots, j_{n-1})$, $J!=j_1!\cdots j_{n-1}!$;
		\item $\partial_x^J=\partial_{x^1}^{j_1} \cdots \partial_{x^{n-1}}^{j_{n-1}}$,  $D_{\xi}^J=D_{\xi_1}^{j_1} \cdots D_{\xi_{n-1}}^{j_{n-1}}$.
		
	\end{itemize}
	We use the Einstein summation convention unless otherwise stated. 
	Throughout the paper, our sign convention is
	\[
	H_1=\langle\vec H_1,-\nu\rangle,
	\]
	so that $H_1>0$ when the mean-curvature vector points into $\Omega$.  
	
	\section{Heat invariants via the Guillemin--Wodzicki  residue}\label{sec: preliminaries}
	
	The heat trace of the Dirichlet-to-Neumann operator has an asymptotic expansion
	\begin{equation}
		\label{eq: Dirichlet-to-Neumann heat trace}
		\Tr(e^{-t\Lambda(\Omega,g)})
		=
		\sum_{j=0}^\infty e^{-t\mu_j(\Omega,g)}
		\sim
		\sum_{j=0}^\infty a_jt^{-n+1+j}
		+
		\sum_{k=1}^\infty b_kt^k\log t	 
	\end{equation}
	as $t \to 0^+$,  which is a special case of the heat expansion for positive elliptic classical pseudodifferential operators \cite[Corollary 2.2$'$]{duistermaatSpectrumPositiveElliptic1975}.  All the  coefficients in \eqref{eq: Dirichlet-to-Neumann heat trace} are spectral invariants.  The coefficients $a_0,\ldots,a_{n-1}$ and all logarithmic coefficients $b_k$ are local in the sense that they are integrals over $\Sigma$ of universal expressions in finite jets of the metric near the boundary.
	
	The   coefficients $a_j$ have been studied systematically using Seeley's calculus; see, among others, \cite{seeleyComplexPowersElliptic1967,polterovichHeatInvariantsSteklov2015,liuAsymptoticExpansionTrace2015,wangRelativeHeatInvariants2019}.  For a three-dimensional manifold of constant sectional curvature $\kappa$, the first three coefficients are 
	\begin{align}
		\label{eq: a0}
		a_0
		&=\frac{1}{2\pi}\Area(\Sigma,g),\\
		\label{eq: a1}
		a_1
		&=\frac{1}{4\pi}\int_\Sigma H_1\,|d\sigma_g|,\\
		\label{eq: a2}
		a_2
		&=\frac{1}{16\pi}\int_\Sigma H_1^2\,|d\sigma_g|
		+\frac{1}{24}\chi(\Sigma)
		+\frac{3\kappa}{16\pi}\Area(\Sigma,g).
	\end{align}
	Here $H_1=(\lambda_1+\lambda_2)/2$ is the normalized mean curvature, $H_2=\lambda_1\lambda_2$ is the second normalized mean curvature, and $\lambda_1,\lambda_2$ are the principal curvatures with the sign convention specified above.  Thus $a_0,a_1,a_2$ detect the boundary area, the total mean curvature, and a combination of the Willmore energy and the Euler characteristic.  These invariants are fundamental, but by themselves they do not provide enough positivity to control a possibly disconnected boundary in all three space forms.
	
	In contrast, the logarithmic coefficients have received much less attention.  A direct computation of $b_k$ through the standard resolvent-parametrix construction requires a lengthy recursive symbol calculation.  Even in dimension three,  computing $b_1$ and $b_2$ in this way already involves several orders of the full symbol and quickly becomes unwieldy.  Our approach replaces
	that calculation with the \emph{Guillemin--Wodzicki noncommutative residue},
	introduced independently by Wodzicki and Guillemin
	\cite{wodzickiLocalInvariantsSpectral1984,guilleminNewProofWeyl1985}. For applications of the Guillemin--Wodzicki residue in spectral geometry, see  \cite{guilleminNewProofWeyl1985,zelditchLecturesWaveInvariants1999}; for generalizations and applications in other settings, see, for example, \cite{connesLocalIndexFormula1995,guilleminResidueTracesCertain1993,fedosovNoncommutativeResidueManifolds1996,pongeNoncommutativeResidueHeisenberg2007}. In \cite{dabrowskiSpectralMetricEinstein2023}, the Guillemin--Wodzicki residues of specific powers of the Laplacian are computed.

	We recall the construction and the two properties that are relevant here.  Let $M$ be a closed $d$-dimensional manifold and let $A\in\Psi_{\mathrm{cl}}^m(M)$ be a classical pseudodifferential operator with full symbol
	\[
	\sigma(\phi,A)(x,\xi) \sim \sum_{j=0}^\infty \sigma_{m-j}(\phi,A)(x,\xi) 
	\]
	in a local coordinate system $\phi$, where $\sigma_{m-j}(\phi,A)(x,\xi)$ is homogeneous of degree $m-j$ in $\xi$ for large $|\xi|$. The {\em Guillemin--Wodzicki residue density} of $A$ is defined as
	\begin{equation}
		\label{eq: residue density introduction}
		\operatorname{res}_x(A) 	:=  \left(\int_{|\xi|=1} \sigma_{-d}(\phi,A)(x,\xi)|d\sigma_\xi| \right)|dx|,
	\end{equation}
	and the   {\em Guillemin--Wodzicki residue} of $A$ is 
	\begin{equation}
		\Res(A):=\int_M\operatorname{res}_x(A).
	\end{equation}
	Our convention in \eqref{eq: residue density introduction} omits the factor
	$(2\pi)^{-d}$ that is included in some definitions of the noncommutative
	residue.  Although the formula is written in coordinates, $\operatorname{res}_x(A)$
	is a globally defined density, independent of the coordinates and of the
	auxiliary norm used to specify $|\xi|=1$
	\cite{wodzickiNoncommutativeResidueFundamentals1987}. 
	Moreover,
	\[
	\Res([A,B])=0.
	\]
	Thus $\Res$ is a trace on the algebra of classical pseudodifferential operators  modulo smoothing operators. In this sense, the single homogeneous symbol term of degree $-d$ extracts the obstruction to extending the ordinary operator trace to the full classical pseudodifferential calculus.  Equivalently, up to the same normalization, it is the coefficient of the logarithmic divergence obtained by restricting the Schwartz kernel to the diagonal.  The same trace also plays a central role in noncommutative geometry, notably in the local index formula \cite{connesLocalIndexFormula1995}.

	The name \emph{residue} also has a literal spectral meaning.  If $Q$ is a positive invertible elliptic classical pseudodifferential operator of order $q>0$, then $\Tr(AQ^{-s})$, initially defined for $\Re s$ sufficiently large, extends meromorphically to the complex plane and
	\begin{equation}
		\label{eq: residue-zeta bridge}
		\underset{s=0}{\operatorname{res}} 	\Tr(AQ^{-s})
		= 	\frac{1}{q(2\pi)^d}\Res(A).
	\end{equation}
	This identity is the bridge from the local symbol in \eqref{eq: residue density introduction} to heat and zeta asymptotics.  It also explains why the noncommutative residue is particularly effective here: a logarithmic heat coefficient is encoded by an ordinary power of $\Lambda$, not by a resolvent of $\Lambda$.

	There is a minor point concerning the zero Steklov eigenvalue.  Let $\Pi_0$ be the orthogonal projection onto $\ker\Lambda$ and put $\widetilde\Lambda=\Lambda+\Pi_0$.  Then $\widetilde\Lambda$ is invertible and $\widetilde\Lambda^k-\Lambda^k$ is smoothing for every positive integer $k$. With the reduced zeta function
	\[ \zeta_\Lambda(s):=\sum_{\mu_j>0}\mu_j^{-s}, \]
	the Mellin transform of \eqref{eq: Dirichlet-to-Neumann heat trace}, together with \eqref{eq: residue-zeta bridge}, gives \cite{wodzickiNoncommutativeResidueFundamentals1987}
	\begin{align}
		\label{eq: a_j as residue}
		a_j	&=\frac{\Gamma(d-j)}{(2\pi)^d}
		\Res\bigl(\widetilde\Lambda(\Omega,g)^{\,j-d}\bigr), 	&&j=0,1,\ldots,d-1,\\
		\label{eq: bk-residue-introduction}
		b_k	&=\frac{(-1)^{k+1}}{(2\pi)^d k!}\Res(\Lambda(\Omega,g)^k), 	&&k=1,2,\ldots.
	\end{align}
	By contrast, the constant coefficient is \cite{seeleyComplexPowersElliptic1967}
	\[a_d=\zeta_\Lambda(0)+\dim\ker\Lambda;\]
	in the present connected setting, $\dim\ker\Lambda=1$.  This is why the local coefficient $a_d$ is not covered by the first formula in \eqref{eq: a_j as residue}.

	Formula \eqref{eq: bk-residue-introduction} reduces the computation of $b_k$
	to the homogeneous term of degree $-d$ in the full symbol of $\Lambda^k$.
	Since $d=2$ in the three-dimensional problem, it is enough to identify the
	degree $-2$ terms in the symbols of $\Lambda$ and $\Lambda^2$.

\section{Metric polynomials on Riemannian manifolds with boundary}
\label{sec:metricpolyn}

The purpose of this section is to set up a purely algebraic language for local scalar invariants near the boundary. This separates the problem into two parts: characterizing the possible metric-jet combinations appearing in the residue density, and reducing them to contractions of curvature and the second fundamental form. In what follows, we fix an integer $n\ge2$.

\subsection{The algebra of metric polynomials}
We begin by introducing the algebras that encode the metric jets appearing in the full symbol of $A_{x^n}$ and in the residue density. Let $\mathcal G$ be the set of formal variables 
\[\mathcal G=\{\mathsf g_{\alpha\beta,c_1\dots c_w}\mid1\le\alpha,\beta\le n-1,1\le c_1,\dots,c_w\le n, w\ge 0\}\]
subject to the symmetry
\[\mathsf g_{\alpha\beta,c_1\dots c_w} = \mathsf g_{\alpha'\beta',c_1'\dots c_w'}\]
whenever $\{\alpha,\beta\}=\{\alpha',\beta'\}$ and $\{c_1,\dots, c_w\}= \{c_1',\dots, c_w'\}$ as multisets.  
These formal variables are modeled on the derivatives of the metric tensor. We also need another formal variable $\mathsf G$ which is modeled on the square root of the determinant of the metric matrix. 
Let $\mathcal R$ denote the algebra $\mathbb R[\mathcal G,\mathsf G]/(\mathsf G^2 - \det\mathsf g)$. We define the algebra $\mathcal P$ as the localization of $\mathcal R$ at $\mathsf G$, i.e. 
\[\mathcal P = \left(\mathbb R[\mathcal G,\mathsf G]/(\mathsf G^2 - \det\mathsf g)\right)_{\mathsf G},\]
and call elements of $\mathcal P$ {\em metric polynomials}.  They are polynomials in the variables from $\mathcal G$ together with $\mathsf G$ and $\mathsf G^{-1}$.    
There is a natural embedding of $\mathcal R$ into $\mathcal P$ because $\mathcal R$ is an integral domain (see Appendix \ref{apd: invariance}). We will not distinguish elements in $\mathcal R$ and their image in $\mathcal P$. 
The matrix $\mathsf g$ is invertible in $\mathrm{Mat}_{(n-1)\times(n-1)}(\mathcal P)$ and the inverse is given by $\mathsf g^{-1}:=(\mathsf g^{\alpha\beta})$ with $\mathsf g^{\alpha\beta} = \mathsf G^{-2}\mathsf g^*_{\alpha\beta}$, where $(\mathsf g^*_{\alpha\beta})$ is the adjugate matrix of $\mathsf g$.

For $c = 1,\dots,n$ and $\mathsf g_{\alpha\beta,c_1\dots c_w}\in\mathcal G$, if we define
\[\partial_c\mathsf g_{\alpha\beta,c_1\dots c_w} := \mathsf g_{\alpha\beta,c_1\dots c_wc}\quad\text{and}\quad \partial_c\mathsf G = \frac{1}{2}\mathsf G^{-1}\partial_c\det\mathsf g,\]
then $\partial_c$ extends to a derivation on $\mathbb R[\mathcal G,\mathsf G]$ by linearity and the Leibniz rule. Since $\partial_c$ preserves the ideal $(\mathsf G^2-\det\mathsf g)$, it induces a derivation on $\mathcal R$. Finally, if we set
\[\partial_c(\mathsf G^{-1}) = -\mathsf G^{-2}\partial_c\mathsf G,\]
then $\partial_c(\mathsf G\mathsf G^{-1}-1)=0$ and hence $\partial_c$ further extends to a derivation on $\mathcal P$.

The algebra $\mathcal P$ is naturally related to polynomial expressions in a Riemannian metric and its derivatives as follows. Let $(\Omega,g)$ be a Riemannian manifold with smooth boundary. If $q\in\partial\Omega$, let $\gamma_q(t)$ be the geodesic starting from $q$ with $\dot\gamma_q(0)$ equal to the inward unit normal. We say $(\Omega,g)$ admits an $\varepsilon$-collar if $(q,t) \mapsto \gamma_q(t)$ defines a diffeomorphism from $\partial\Omega\times[0,\varepsilon)$ onto a neighborhood of $\partial\Omega$, called an  {\em $\varepsilon$-collar}. By the collar neighborhood theorem, every compact Riemannian manifold with smooth boundary admits an $\varepsilon$-collar for $\varepsilon$ small enough.

For any open subset $U$ of $\partial\Omega$, let $U^\varepsilon$ be the image of $U\times[0,\varepsilon)$ under the above diffeomorphism. If $(x^1,\dots,x^{n-1})$ is a coordinate system on $U$, then, with $x^n=t$, the coordinates $(x^1,\dots,x^n)$ defines a coordinate system on $U^\varepsilon$, called the {\em boundary normal coordinate system associated with} $(x^1,\dots, x^{n-1})$. In such coordinate systems, the Riemannian metric takes the form
\begin{equation*}
    g(p) = g_{\alpha\beta}(p)dx^\alpha\otimes dx^\beta + dx^n\otimes dx^n,\quad\forall p\in U^\varepsilon.
\end{equation*}
For $x^n\in[0,\varepsilon)$, let $g_{x^n}$ denote the pull-back metric on $\partial\Omega$ under the smooth map $q\mapsto\gamma_q(x^n)$. Write $|d\sigma_{x^n}|$ for the area density of the metric $g_{x^n}$.

Now we associate to any abstract element in the algebra $\mathcal P$ a concrete polynomial in entries of the Riemannian metric and their derivatives.  
Let $\phi = (x^1,\dots,x^{n-1})$ be a local coordinate system defined on $U\subset\partial\Omega$. We construct an evaluation map $\mathcal P\to C^\infty(U\times[0,\varepsilon))$, $\mathsf p\mapsto\mathsf p(\phi,g)$ by setting
\[\mathsf g_{\alpha\beta,c_1\dots c_w} \mapsto \partial_{x^{c_w}}\cdots\partial_{x^{c_1}}(g_{x^n})_{\alpha\beta},\quad\mathsf G\mapsto\sqrt{\det((g_{x^n})_{\alpha\beta})}\] 
where $(g_{x^n})_{\alpha\beta} = g_{x^n}(\partial_{x^\alpha},\partial_{x^\beta})$.  
This association is faithful in the following sense:
\begin{proposition}\label{prop: faithful}
	Let $\mathsf p\in\mathcal P$. If for any compact Riemannian manifold $(\Omega,g)$  with an $\varepsilon$-collar and any boundary normal coordinate system $\phi$ on $U \subset \partial \Omega$,  $\mathsf p(\phi,g) = 0$, then $\mathsf p = 0$.
\end{proposition}
See Appendix \ref{apd: invariance} for the proof.

\subsection{Invariant polynomials}\label{subsec: exp of inv poly}
 
Let $\mathsf p\in\mathcal P$. 
Consider the family of Riemannian metrics $g_{x^n}$ parametrized by $x^n$ on $\Sigma=\partial \Omega$.
We say  $\mathsf p\in\mathcal P$ is {\em invariant} if for any $(q,t) \in \Sigma \times [0,\varepsilon)$, and  any geodesic normal coordinate chart $\phi=(x^1,\dots,x^{n-1})$  of $(\Sigma, g_{t})$ centered at $q$, the evaluation $\mathsf p(\phi,g)(q)$ depends only on $q$ and $t$ and is independent of the choice of $\phi$. In this case we denote $\mathsf p(\phi,g)$ by $\mathsf p(g)$.

We list several invariant polynomials that will be used later.
\begin{enumerate}
	\item  The metric polynomial $\mathsf G$ is invariant because $\mathsf G(g) \equiv \sqrt{\det(\delta_{\alpha\beta})} = 1$ for any metric $g$.
	\item As in Riemannian geometry, we set 
	\[\mathsf\Gamma_{ab}^c = \frac{1}{2}\mathsf g^{cd}(\mathsf g_{db,a} + \mathsf g_{ad,b} - \mathsf g_{ab,d})\]
	and 
	\[\mathsf{Rm}_{abcd} = \mathsf g_{dp}\left(\partial_b\mathsf\Gamma_{ac}^p-\partial_a\mathsf\Gamma_{bc}^p + \mathsf\Gamma_{ac}^q\mathsf\Gamma_{bq}^p - \mathsf\Gamma_{bc}^q\mathsf\Gamma_{aq}^p\right).\]
	Here we assume $\mathsf g_{an,c_1\dots c_w} = 0$ unless $a = n$ and $w=0$, in which case $\mathsf g_{an,c_1\dots c_w} = 1$. These coordinate components are not invariants.  However, the following   metric polynomials
	\begin{itemize}
		\item $\mathsf R = \mathsf g^{ac}\mathsf g^{bd}\mathsf{Rm}_{abcd}$
		\item $\mathsf {Ric}^\bot = \mathsf g^{ab}\mathsf{Rm}_{anbn}$
	\end{itemize} 
	are invariants: after evaluating on a Riemannian manifold, they  give the scalar curvature and the Ricci curvature in the normal direction.

	Similarly, one can define the tangential Christoffel symbols  $(\mathsf\Gamma^\top)_{\alpha\beta}^\gamma$, and curvature components $\mathsf{Rm}^\top_{\alpha\beta\gamma\delta}$; then the resulting polynomial 
	\begin{itemize}
		\item $\mathsf R^\top =  \mathsf g^{\alpha\gamma}\mathsf g^{\beta\delta}\mathsf{Rm}^\top_{\alpha\beta\gamma\delta}$
	\end{itemize}  
	is invariant and represents the scalar curvature of $(\partial\Omega,g_{x^n})$.
	\item Set $\mathsf h = (\mathsf h_{\alpha\beta}) := ( -\frac{1}{2}\mathsf g_{\alpha\beta,n})\in\mathrm{Mat}_{(n-1)\times(n-1)}(\mathcal P)$ and assume
	\[\det(tI-\mathsf g^{-1}\mathsf h) = \sum_{k=0}^{n-1}(-1)^k\binom{n-1}{k}\mathsf H_{k}t^{n-1-k},\]
	where $n=\dim \Omega$. 
	Then each $\mathsf H_k$ is invariant and represents the pull-back to $\Sigma$ of the $k$-th order mean curvature of $ \{x^n = t\}$ in $(\Omega, g)$.

\item	 The same idea applies to any tensor expression in Riemannian geometry. Let $k\ge 0$ and let $\{\mathsf p_{a_1\dots a_k}|1\le a_1,\dots,a_k\le n\}\subset\mathcal P$ be a family of metric polynomials. We may define their {\em covariant derivatives} by
	\begin{equation*}
		\label{eq: def of cov diff}
		\nabla\mathsf p_{a_1\dots a_kc} := \partial_{c}\mathsf p_{a_1\cdots a_k}-\sum_{j=1}^k\sum_{b=1}^n\mathsf\Gamma_{ca_j}^{b}\mathsf p_{a_1\cdots a_{j-1}ba_{j+1}\cdots a_k},
	\end{equation*}
	and write $\mathsf p_{a_1\dots a_k;c_1\dots c_w}$ for
	\[\nabla(\cdots(\nabla\mathsf p)\cdots)_{a_1\dots a_kc_1\dots c_w}.\]
	Similarly, by restricting all subscripts to range from $1$ to $n-1$, one can define {\em covariant derivatives along the boundary}, for which we write $\mathsf p_{\alpha_1\dots\alpha_k:\gamma_1\dots \gamma_w}$ to distinguish them. The semicolon `;' indicates the ``ambient covariant derivatives'', while the colon `:' indicates the induced ``tangential covariant derivatives''. For example, one may define the tangential Hessian to be 
	\[
	\mathrm{Hess}^\top\mathsf p_{\alpha\beta} := \mathsf p_{:\alpha\beta}.
	\]
	It follows immediately that
	\begin{itemize}
		\item If $\mathsf p$ is invariant, so is $\mathsf{\Delta^\top p} := \mathsf g^{\alpha\beta}\mathrm{Hess}^\top\mathsf p_{\alpha\beta}$.
	\end{itemize}
	
\item	For two families $\{\mathsf p_{\alpha_1\dots \alpha_k}^{\beta_1\dots \beta_l}|1\le \alpha_1, \dots, \alpha_k,     \beta_1,\dots,\beta_l \le n-1\}$ and $\{\mathsf q_{\alpha_1\dots \alpha_k}^{\beta_1\dots \beta_l}|1\le \alpha_1,\dots,\alpha_k,\beta_1,\dots,\beta_l\le n-1\}$  of metric polynomials, we set
	\[\langle\mathsf p,\mathsf q\rangle = \mathsf g^{\alpha_1\gamma_1}\cdots \mathsf g^{\alpha_k\gamma_k}\mathsf g_{\beta_1\delta_1}\cdots \mathsf g_{\beta_l\delta_l}\mathsf p_{\alpha_1\dots \alpha_k}^{\beta_1\dots \beta_l}\mathsf q_{\gamma_1\dots \gamma_k}^{\delta_1\dots \delta_l}.\]   
    With this notation, we can write down more invariant polynomials that will appear in Theorem \ref{thm: main thm b1 b2}:
	\begin{itemize} 
    \item If $\mathsf p$ is invariant, so is  $\langle \nabla \mathsf p, \nabla \mathsf p\rangle$,
	\item $\langle\mathsf h,\mathsf {Rm}^\bot\rangle$, $\langle\mathsf h, \mathrm{Hess}^\top\mathsf p\rangle$, $\langle\dvg^\top \mathsf h,\nabla^\top\mathsf H_1\rangle$, $\langle\mathsf h^2,\mathsf{Rm}^\bot\rangle$, $\langle\mathsf{Rm}^\bot,\mathsf{Rm}^\bot\rangle$, $\langle\mathsf h,\nabla_n\mathsf{Rm}^\bot\rangle$.
    \end{itemize}
 where 
 \[
 \mathsf {Rm}^\bot_{\alpha\beta} := \mathsf{Rm}_{\alpha n\beta n}, \quad  \mathsf h_{\alpha\beta}^2 := \mathsf g^{\gamma\delta}\mathsf h_{\alpha\gamma}\mathsf h_{\beta\delta}, \quad \text{and}\quad    (\dvg^\top\mathsf h)_\alpha := \mathsf g^{\beta\gamma}\mathsf h_{\alpha\beta:\gamma}.
 \]
  \end{enumerate} 

By Proposition \ref{prop: faithful}, if an identity holds in any local coordinate system, then it holds in the formal algebra as well. In particular, in the formal algebra $\mathcal P$ we also have the Gauss--Codazzi equations 
\begin{align}
    \label{eq: Gauss equation}
    \mathsf{Rm}^\top_{\alpha_1\beta_1\alpha_2\beta_2} & = \mathsf{Rm}_{\alpha_1\beta_1\alpha_2\beta_2} + \mathsf h_{\alpha_1\alpha_2}\mathsf h_{\beta_1\beta_2} - \mathsf h_{\alpha_1\beta_2}\mathsf h_{\alpha_2\beta_1},\\
    \label{eq: Codazzi equation}
    \mathsf{Rm}_{\alpha\beta\gamma n} & = \mathsf h_{\beta\gamma:\alpha} - \mathsf h_{\alpha\gamma:\beta},
\end{align}
and the Bianchi identities 
\begin{align}
    \label{eq: 1st Bianchi equation}
    \mathsf{Rm}_{abcd} + \mathsf{Rm}_{acdb} + \mathsf{Rm}_{adbc} & = 0,\\
    \label{eq: 2nd Bianchi equation}
    \mathsf{Rm}_{abcd;e} + \mathsf{Rm}_{abde;c} + \mathsf{Rm}_{abec;d} & = 0.
\end{align}

\subsection{The quotient algebra $\mathcal O$}
Let $\mathcal O$ denote the quotient algebra $\mathcal P/\mathcal J$, where $\mathcal J$ is the ideal generated by
\begin{equation}\label{eq: kernel of P to O}
    \mathsf G-1,\quad\mathsf g_{\alpha\beta}-\delta_{\alpha\beta},\quad \mathsf g_{\alpha\beta,\gamma},\quad \alpha,\beta,\gamma = 1,\dots,n-1.
\end{equation}

Let $A = (A_{\alpha\beta})\in\Ogroup{n-1}$. For $\mathsf g_{\alpha\beta,\gamma_1\dots\gamma_tn\dots n}$,  we set 
\[A.\mathsf g_{\alpha\beta,\gamma_1\dots\gamma_tn\dots n}  = \sum_{1\le\alpha',\beta',\gamma_1',\dots,\gamma_t'\le n-1} A_{\alpha\alpha'}A_{\beta\beta'}A_{\gamma_1\gamma_1'}\dots A_{\gamma_t\gamma_t'}\mathsf g_{\alpha'\beta',\gamma_1'\dots\gamma_t'n\dots n}\]
and
\[A.\mathsf G = \mathsf G.\]
Because
\[\sum_{1\le\alpha'\le n-1} A_{\alpha\alpha'} B_{\alpha'\alpha''} =  (AB)_{\alpha\alpha''},\]
for any $A, B\in\Ogroup{n-1}$, this defines an $\Ogroup{n-1}$-action on 
$\mathbb R[\mathcal G,\mathsf G]$. Moreover, since $A.(\mathsf G^2 - \det\mathsf g) = \mathsf G^2 - \det(A)^2\det\mathsf g = \mathsf G^2 - \det\mathsf g$, this induces an $\Ogroup{n-1}$-action on $\mathcal R$, which extends to an action on $\mathcal P$ by setting $A.\mathsf G^{-1} = \mathsf G^{-1}$. Finally, since
\[\sum_{1\le\alpha',\beta'\le n-1}A_{\alpha\alpha'}A_{\beta\beta'}\delta_{\alpha'\beta'} = \delta_{\alpha\beta},\]
the ideal $\mathcal J$ is invariant under the $\Ogroup{n-1}$-action and hence the action descends to an $\Ogroup{n-1}$-action on $\mathcal O$.

The following proposition converts the invariance of metric polynomials into a purely algebraic property, so that we may apply  Weyl's invariant theory \cite{weyl1997classical} in what follows.
\begin{proposition}
    \label{prop: evaluation invariance is algebraic invariance}
    Let $\mathsf p\in\mathcal P$. Then $\mathsf p$ is invariant if and only if its image in $\mathcal O$ is fixed by the $\Ogroup{n-1}$-action.
\end{proposition}
Let $\mathcal Q$ denote the $\mathbb R$-subalgebra of $\mathcal P$ generated by
\[\mathsf{Rm}_{a_1b_1a_2b_2;c_1\cdots c_k},\ \mathsf h_{\alpha\beta;\gamma_1\cdots \gamma_l}\quad k, l\ge 0\]
of $\mathcal P$. A classical theorem in Riemannian geometry can be restated in the following way:

\begin{proposition}[{\cite[Theorem 1.1.3]{gilkeyAsymptoticFormulaeSpectral2004}}]
    Every element of $\mathcal O$ has a representative in $\mathcal Q$.
\end{proposition}

In particular, we have
\begin{align}
    \label{eq: g (2,0)}
    \mathsf g_{\alpha\beta,\gamma\delta} & \equiv \frac{1}{3}(\mathsf{Rm}^\top_{\alpha\gamma\delta\beta} + \mathsf{Rm}^\top_{\alpha\delta\gamma\beta})\mod\mathcal J, \\
    \label{eq: g (0,2)}
    \mathsf g_{\alpha\beta,nn} & = 2(\mathsf g^{\gamma\delta}\mathsf h_{\alpha\gamma}\mathsf h_{\beta\delta} - \mathsf {Rm}_{\alpha\beta}^\bot),\\
    \label{eq: g (0,3)}
    \mathsf g_{\alpha\beta,nnn} & = 4\mathsf g^{\gamma\delta}\bigl(\mathsf h_{\alpha\gamma}\mathsf {Rm}^\bot_{\beta\delta}+\mathsf h_{\beta\gamma}\mathsf {Rm}^\bot_{\alpha\delta}\bigr)-2\mathsf {Rm}^\bot_{\alpha\beta;n},\\
    \label{eq: g (2,1)}
    \mathsf g_{\alpha\beta,\gamma\delta n} & \equiv -2\mathsf h_{\alpha\beta:\gamma\delta} + 2\sum_\epsilon\mathsf h_{\beta\epsilon}\mathsf g_{\gamma\alpha,\epsilon\delta} + 2\sum_\epsilon\mathsf h_{\alpha\epsilon}\mathsf g_{\gamma\beta,\epsilon\delta}\mod\mathcal J.
\end{align}

We will use the following lemma to reduce computations in $\mathcal P$ to computations in $\mathcal O$:

\begin{lemma}
    \label{lem: independent of phi}
    Let $\mathsf p\in\mathcal J\subset\mathcal P$. If for any Riemannian manifold $(\Omega,g)$ with an $\varepsilon$-collar, the evaluation $\mathsf p(\phi,g)(q,t)$ is independent of $\phi$, for all $(q,t)\in\partial\Omega\times[0,\varepsilon)$, then $\mathsf p = 0$.
\end{lemma}

See Appendix \ref{apd: invariance} for the proof of Proposition \ref{prop: evaluation invariance is algebraic invariance} and \ref{lem: independent of phi}. See Appendix \ref{apd: computation} for the proof of formulas \eqref{eq: g (2,0)}--\eqref{eq: g (2,1)}.

\subsection{The algebra $\mathcal{SP}$ and its subspaces}
Given any formal variable $\mathsf g_{\alpha\beta,c_1\dots c_w}$, we define its tangential order and normal order as 
\[\aligned 
\ord^\top(\mathsf g_{\alpha\beta,c_1\dots c_w})& =\#\{j|1\le j\le w,c_j<n\},
\\
\ord^\bot(\mathsf g_{\alpha\beta,c_1\dots c_w})&=\#\{j|1\le j\le w,c_j=n\}
\endaligned\]
respectively. The orders of a monomial are then defined naturally. Let $\mathcal P_{w^\top,w^\bot}$ be the space spanned by monomials of bi-degree $(w^\top,w^\bot)$. Then 
\[  \mathcal P_{w_1^\top,w_1^\bot}\cdot\mathcal P_{w_2^\top,w_2^\bot}\subset \mathcal P_{w_1^\top + w_2^\top,w_1^\bot + w_2^\bot}.\]
This equips the algebra $\mathcal P$ with a bi-grading  
\[\mathcal P = \bigoplus_{w^\top,w^\bot\in\mathbb N_0}\mathcal P_{w^\top,w^\bot}.\]

Since the symbol of a pseudodifferential operator is a function on the cotangent bundle, we introduce 
$\eta_1,\dots,\eta_{n-1},\rho$
to encode cotangent variables and their norm. Let $\mathcal S$ be the $\mathbb R$-algebra generated by formal variables $\eta_1,\dots,\eta_{n-1},\rho$, localized at $\rho$, i.e.
\[\mathcal S = \mathbb R[\eta_1,\dots,\eta_{n-1},\rho,\rho^{-1}]/(\rho\rho^{-1}-1).\]

For any $\gamma=1,\dots,n-1$, define a derivation  $D_\gamma: \mathcal S\otimes\mathcal P \to \mathcal S\otimes\mathcal P$ by  letting  $D_\gamma=0$ on $\mathcal P$ and letting 
\[D_\gamma\eta_\alpha = \delta_{\alpha\gamma},\quad D_\gamma\rho = -\rho^{-1}\eta_\alpha\mathsf g^{\alpha\gamma}.\] 
Similarly,   extend the derivation $\partial_c$ on $\mathcal P$ to  $\partial_c: \mathcal S\otimes\mathcal P \to \mathcal S\otimes\mathcal P$ by  letting 
\[\partial_c\rho = -\frac{1}{2}\rho^{-1}\eta_{\alpha}\eta_{\beta} \partial_c\mathsf g^{\alpha\beta},\quad \partial_c \eta_\alpha = 0.\]

For $m\in\mathbb Z$ and $k \in \mathbb N_0$, we denote
\[\mathcal S^{m,k}   = \Span_\mathbb R\{\rho^m\eta^J\colon|J| = k\}.\]
To describe the full symbol of the Dirichlet-to-Neumann operator, for any   $(w^\top,w^\bot)\in\mathbb N_0^2$ we set
\begin{align*} 
    \mathcal S\mathcal P_{w^\top,w^\bot}^m   = \bigoplus_{0\le k\le w^\bot}\mathcal S^{m-k,k}\otimes\mathcal P_{w^\top + k,w^\bot - k}.
\end{align*}
The following proposition shows that these spaces together form a graded $\mathbb R$-algebra $\mathcal{SP}$, which will be useful in the computations below. Note that $\mathcal {SP}$ is isomorphic to $\mathcal S\otimes\mathcal P$ as a $\mathbb R$-algebra, but equipped with a grading.
\begin{proposition}\label{prop:thespaceSP}
	\begin{enumerate}
		\item $\mathcal{SP}_{w_1^\top,w_1^\bot}^{m_1}\cdot\mathcal{SP}_{w_2^\top,w_2^\bot}^{m_2}\subset\mathcal{SP}_{w_1^\top+w_2^\top,w_1^\bot+w_2^\bot}^{m_1+m_2}$.
		\item For any multi-index $J$, the ``tangential derivations” satisfy 
		\begin{equation*}
			D^J(\mathcal{SP}_{w^\top,w^\bot}^m) \subset \mathcal{SP}_{w^\top-|J|,w^\bot+|J|}^{m-|J|},\quad \partial^J(\mathcal{SP}_{w^\top,w^\bot}^m) \subset \mathcal{SP}_{w^\top+|J|, w^\bot}^m,
		\end{equation*}
		\item The ``normal derivation” satisfies   
		\begin{align*}
			\partial_n(\mathcal{SP}_{w^\top,w^\bot}^m) \subset \mathcal{SP}_{w^\top,w^\bot+1}^m.
		\end{align*}
	\end{enumerate}
\end{proposition}

\begin{proof} 
(1) follows from the definition.  In what follows, $\mathbf 1_\alpha$ denotes the multi-index with a $1$ in the $\alpha$-th position and zeros elsewhere. Let 
	\[\mathsf s = \rho^{m-|J|}\eta^J\mathsf p\in\mathcal{SP}_{w^\top,w^\bot}^m.\] For any $\gamma=1,\cdots,n-1$, one has
	\begin{align*}
		D_{\gamma}\mathsf s
		& =  -(m -|J|)\rho^{ m - 2 - |J|} \eta_\alpha\eta^J(\mathsf g^{\alpha\gamma}  \mathsf p) +  j_\gamma \rho ^{m - |J|}\eta_\gamma^{-1}\eta^J \mathsf p\\
		& =  -(m -|J|)\rho ^{ m - 1 - |J+\mathbf 1_\alpha|}\eta^{J+\mathbf 1_\alpha}(  \mathsf g^{\alpha\gamma}\mathsf p) + j_\gamma \rho ^{m -1 - |J-\mathbf 1_\gamma|}\eta^{J-\mathbf 1_\gamma}\mathsf p.
	\end{align*}
	For $c=1,\dots,n$, one has
	\begin{align*}
		\partial_{c} \mathsf s 
		& = -\frac{m -|J|}{2}\rho ^{ m - 2 - |J|} \eta_\beta \eta_\gamma\eta^J (\partial_c  \mathsf g^{\beta\gamma}\mathsf p) + \rho^{m -|J|}\eta^J (\partial_c\mathsf p)\\
		& = -\frac{m -|J|}{2}\rho^{m - |J + \mathbf 1_\beta + \mathbf 1_\gamma|}\eta^{J + \mathbf 1_\beta + \mathbf 1_\gamma}(\partial_c\mathsf g^{\beta\gamma}\mathsf p) + \rho^{m -|J|}\eta^J (\partial_c\mathsf p)
	\end{align*}
	Note that $\mathsf g^{\alpha\beta} = \mathsf G^{-2}  \mathsf g_{\alpha\beta}^*\in\mathcal P_{0,0}$. After relabeling indices, we conclude from this computation that 
	\begin{equation}
		\label{eq: D_gamma and partial_gamma}
		D_\gamma(\mathcal{SP}_{w^\top,w^\bot}^m) \subset \mathcal{SP}_{w^\top-1,w^\bot+1}^{m-1},\quad \partial_\gamma(\mathcal{SP}_{w^\top,w^\bot}^m) \subset \mathcal{SP}_{w^\top+1,w^\bot}^m,
	\end{equation}
	and $\partial_n(\mathcal{SP}_{w^\top,w^\bot}^m) \subset \mathcal{SP}_{w^\top,w^\bot+1}^m$, which gives (3).
	Repeated use of \eqref{eq: D_gamma and partial_gamma} gives (2). 
\end{proof}

 \subsection{Invariants of given type}

 We will use a refined notion of the grading on $\mathcal P$. Given  $r\in\mathbb Z$ and any variables $\mathsf m_j$ ($1 \le j \le d$) in $\mathcal G$, define the {\em  type} of the metric polynomial  $\mathsf p = \mathsf G^{-r}\prod_{j=1}^d\mathsf m_j\in\mathcal P$ to be the multiset 
 \[T = \{(\ord^\top(\mathsf m_j),\ord^\bot(\mathsf m_j))\mid j=1,\dots,d\}\] 
 of pairs, and define the type of a scalar as $(0,0)$. Let $\mathcal P_T$ denote the subspace of $\mathcal P$ spanned by all the polynomials of type $T$. It is straightforward to verify that if $T \neq \widetilde T$, then   $\mathcal P_T \cap \mathcal P_{\widetilde T} = \{0\}$. Therefore we have a direct sum decomposition of $\mathcal P$ into $\mathbb R$-vector spaces $\mathcal P = \bigoplus_{T\in\mathcal T}\mathcal P_T$, where $\mathcal T$ denotes the set of all types.   Let $\Proj_T$ be the projection from $\mathcal P$ onto $\mathcal P_T$. 
 This construction induces a decomposition of $\mathcal S$-modules $\mathcal {SP}\simeq\mathcal S\otimes \mathcal P = \bigoplus_{T\in\mathcal T}\mathcal S\otimes\mathcal P_T$. The resulting projections are also denoted by $\Proj_T$. Moreover, each $\mathcal P_T$ is invariant under the $\Ogroup{n-1}$-action since the action preserves the grading of $\mathcal P$.

For any type $T$, let $\mathcal O_T$ be the image of   $\mathcal P_T$ under the quotient $\mathcal P \to \mathcal O=\mathcal P/\mathcal J$. 
For $w\in\mathbb N$, let $\widetilde{\mathcal P}_w$ be the subspace $\bigoplus_{0\le j\le [w/2]}\mathcal P_{2j,w-2j}$, and let  $\widetilde{\mathcal O}_w$ denote the image of $\widetilde{\mathcal P}_w$ under the quotient map $\mathcal P\to\mathcal O$.
When $n=3$ and $w=3$, the direct sum decomposition of $\mathcal P$ induces a decomposition of $\widetilde{\mathcal O}_3$ as 
\[\widetilde{\mathcal O}_3 = \mathcal O_{(0,1)(2,0)} + \mathcal O_{(2,1)} + \mathcal O_{(0,1)(0,1)(0,1)} + \mathcal O_{(0,1)(0,2)} + \mathcal O_{(0,3)}.\]
For a type $T$, let $\mathcal I_T$ denote the space consisting of the images of  all the invariant polynomials in $\mathcal O_T$.
 
\begin{lemma}\label{lem: basis of I_T}
    The space $\mathcal I_T$ is spanned by
 	\begin{enumerate}[(1)]
 		\item $\mathsf H_1\mathsf R^\top$, if $T = (0,1)(2,0)$;
 		\item $\mathsf{\Delta^\top H_1}, \mathsf H_1\mathsf R^\top, \partial_n\mathsf R$, if $T = (2,1)$;
 		\item $\mathsf H_1^3, \mathsf H_1\mathsf H_2$, if $T = (0,1)(0,1)(0,1)$;
 		\item $\mathsf H_1^3,\mathsf H_1\mathsf H_2, \mathsf H_1\mathsf{Ric}^\bot, \langle \mathsf h, \mathsf {Rm}^\bot\rangle$, if $T = (0,1)(0,2)$;
 		\item $\mathsf H_1\mathsf{Ric}^\bot, \mathsf \langle \mathsf h, \mathsf {Rm}^\bot\rangle, \partial_n\mathsf{Ric}^\bot$, if $T = (0,3)$;
 	\end{enumerate}
\end{lemma}
\begin{proof}
 	By Weyl's first main theorem for orthogonal groups \cite[Section 1.7]{gilkeyAsymptoticFormulaeSpectral2004}, for each type $T$, $\mathcal I_T$ is spanned by all possible scalar contractions of certain multi-linear maps. Using \eqref{eq: g (2,0)}--\eqref{eq: g (2,1)}, they are:
 	\begin{enumerate}[(1)]
 		\item $\mathsf h_{\alpha\alpha}\mathsf{Rm}^\top_{\beta\gamma\beta\gamma}$, $\mathsf h_{\alpha\beta}\mathsf{Rm}^\top_{\alpha\gamma\beta\gamma}$, if $T = (0,1)(2,0)$;
 		\item $\mathsf h_{\alpha\alpha:\beta\beta}$,$\mathsf h_{\alpha\beta:\alpha\beta}$, $\mathsf h_{\alpha\beta:\beta\alpha}$, $\mathsf h_{\alpha\alpha}\mathsf{Rm}^\top_{\beta\gamma\beta\gamma}$, $\mathsf h_{\alpha\beta}\mathsf{Rm}^\top_{\alpha\gamma\beta\gamma}$, if $T = (2,1)$;
 		\item $\mathsf h_{\alpha\alpha}\mathsf h_{\beta\beta}\mathsf h_{\gamma\gamma}$, $\mathsf h_{\alpha\beta}\mathsf h_{\alpha\beta}\mathsf h_{\gamma\gamma}$, $\mathsf h_{\alpha\beta}\mathsf h_{\beta\gamma}\mathsf h_{\gamma\alpha}$, if $T = (0,1)(0,1)(0,1)$;
 		\item $\mathsf h_{\alpha\alpha}\mathsf h_{\beta\beta}\mathsf h_{\gamma\gamma}$, $\mathsf h_{\alpha\beta}\mathsf h_{\alpha\beta}\mathsf h_{\gamma\gamma}$, $\mathsf h_{\alpha\beta}\mathsf h_{\beta\gamma}\mathsf h_{\gamma\alpha}$, $\mathsf h_{\alpha\beta}\mathsf {Rm}^\bot_{\alpha\beta}$, $\mathsf h_{\alpha\alpha}\mathsf {Rm}^\bot_{\beta\beta}$, if $T = (0,1)(0,2)$;
 		\item $\mathsf h_{\alpha\beta}\mathsf {Rm}^\bot_{\alpha\beta}$, $\mathsf h_{\alpha\alpha}\mathsf {Rm}^\bot_{\beta\beta}$, $\mathsf{Rm}_{\alpha n\alpha n;n}$, if $T = (0,3)$.
 	\end{enumerate}
 	Here and below, all repeated indices are summed over. Then the lemma follows from \eqref{eq: Gauss equation}-\eqref{eq: 2nd Bianchi equation}.
\end{proof}
Consequently, we obtain
\begin{proposition}\label{prop:decomforn=3}
 	For any invariant $\mathsf p \in \widetilde{\mathcal P}_3$, there exist $C_1,\dots,C_8\in\mathbb R$ such that
 	\begin{equation*}
 		\begin{aligned}
 			\mathsf p \equiv \ & C_1\mathsf H_1^3 + C_2\mathsf H_1\mathsf R^\top + C_3\mathsf H_1\mathsf H_2 + C_4\mathsf {\Delta^\top H_1} + C_5\langle \mathsf h, \mathsf {Rm}^\bot\rangle \\
 			& + C_6\mathsf H_1\mathsf{Ric}^\bot + C_7\partial_n\mathsf{Ric}^\bot + C_8\partial_n\mathsf R\mod\mathcal J.
 		\end{aligned}
    \end{equation*}
\end{proposition}

\section{The full symbol of $A_{x^n}$}\label{sec: full symbol}

\subsection{The Lee--Uhlmann formula}\label{subsec: L-U}

We first write down the Lee-Uhlmann formula for the full symbol of the Dirichlet-to-Neumann operator $\Lambda(\Omega,g)$, which was obtained by Lee and Uhlmann in \cite[Proposition 1.1]{leeDeterminingAnisotropicReal1989}. The idea is to construct a family of pseudodifferential operators $A_{x^n}(\Omega,g)$ depending on the parameter $x^n$ in a boundary normal coordinate system whose symbols can be computed recursively,  such that $A_0(\Omega,g)$ equals $-\Lambda(\Omega,g)$ modulo smoothing operators.  
If we write $\sum_{j\ge 0}\sigma_{m-j}(\phi,A)$ for the full symbol expansion of a pseudodifferential operator $A$ in the local coordinate system $\phi$, then the Lee-Uhlmann formulas for $A_{x^n}$ take the form 
\begin{align}
        \label{eq: r_1}
        \sigma_{1}(\phi,A_{x^n}) & = - \sqrt{q_2},\\
        \label{eq: r_0}
        \sigma_{0}(\phi,A_{x^n}) & = \frac{1}{2\sqrt{q_2}}\left(\sum_\alpha D_{\xi_\alpha} \sigma_{1}\partial_{x^\alpha} \sigma_{1} - q_1 + \partial_{x^n} \sigma_{1} - E \sigma_{1}\right),
\end{align}
and for $j > 1$,
\begin{equation}\label{eq: r_{1 - j}}
    \sigma_{1-j}(\phi,A_{x^n}) = \frac{1}{2\sqrt{q_2}}\left(\sum_{\substack{0\le k,l\le j-1,\\|J| = j-k-l}}\frac{1}{J!}D_\xi^J\sigma_{1-k} \partial_x^J \sigma_{1-l} + \partial_{x^n} \sigma_{2-j} - E\sigma_{2-j}\right),
\end{equation}
where
\begin{align*}
    q_2   = g^{\alpha\beta}\xi_\alpha\xi_\beta \quad \text{and} \quad 
    q_1   = - i\left(g^{\alpha\beta}\partial_{x^\alpha}\log \sqrt{\det(g)} + \partial_{x^\alpha}g^{\alpha\beta}\right)\xi_\beta 
\end{align*}
are the leading  and sub-leading symbols of the Laplace--Beltrami operator on $\Sigma$, and 
\[E = -\partial_{x^n}\log \sqrt{\det(g)} = -|d\sigma_{x^n}|^{-1}\partial_{x^n}|d\sigma_{x^n}|\]
is the normal derivative of the area density.

In fact, if we write $\sigma(\phi,A_{x^n}^2)$ for the full symbol of  $A_{x^n}^2$ in the coordinate system $\phi$, then \eqref{eq: r_1}--\eqref{eq: r_{1 - j}} follow from solving the equation (cf. \cite[(1.5)]{leeDeterminingAnisotropicReal1989}) 
\begin{equation}
	\label{operatordecomp}
	A_{x^n}^2 - Q + i[D_{x^n},A_{x^n}] - EA_{x^n} = 0,
\end{equation}
where $Q = q_2(x,D_{x'}) + q_1(x,D_{x'})$, $D_{x'} = (D_{x^1},\dots D_{x^{n-1}})$. In particular, we have
\begin{equation}\label{eq: sigma A^2 and sigma A}
    \sigma_{-2}(\phi,A_{x^n}^2) + \partial_{x^n}\sigma_{-2}(\phi,A_{x^n})-E\sigma_{-2}(\phi,A_{x^n}) = 0.
\end{equation}
This identity will help us to compute $\res A_{x^n}^2$ from $\res A_{x^n}$.

\subsection{Characterization of the full symbol}

Let $(\Omega^n,g)$ be a Riemannian manifold with an $\varepsilon$-collar, and let $\phi$ be a boundary normal coordinate system defined on $U^\varepsilon$. For any $x^n\in[0,\varepsilon)$ and $(q,\xi_\alpha dx^\alpha)\in T^*U\setminus 0_U$, where $0_U\subset T^*U$ denotes the image of the zero section, we construct an evaluation map $\mathcal{SP}\to C^\infty((T^*U\setminus 0_U)\times[0,\varepsilon))$ by setting
\[\eta_\alpha\mapsto i\xi_\alpha,\quad \rho \mapsto\sqrt{g_{x^n}^{\alpha\beta}(q)\xi_\alpha\xi_\beta},\quad \mathsf p\mapsto \mathsf p_{x^n}(\phi,g)(q),\]
where $(g_{x^n}^{\alpha\beta}) = (g_{x^n}(\partial_{x^\alpha},\partial_{x^\beta}))^{-1}.$ As $(q,\xi)$ varies in $T^*U\setminus 0_U$, this gives rise to a smooth function $\mathsf s(\phi,g)$. Write $\mathsf s(\phi,g,x^n) = \mathsf s(\phi,g)(\cdot,x^n)$. Moreover, it follows from the definition that
\[\partial_c\mathsf s(\phi,g) = \partial_{x^c}(\mathsf s(\phi,g)),\qquad D_\gamma\mathsf s(\phi,g) = -i\partial_{\xi_\gamma}(\mathsf s(\phi,g)).\]

Using the Lee--Uhlmann formula, we prove:

\begin{proposition}\label{prop: characterization of symbol}
    For any $j\in\mathbb N_0$, there exists $\mathsf s_{1-j}\in\mathcal{SP}_{0,j}^{1-j}$ such that  for any Riemannian manifold $(\Omega^n,g)$ with an $\varepsilon$-collar,    one has
    \begin{equation}\label{eq: characterization of symbol}
        \sigma_{1-j}(\phi,A_{x^n}(\Omega,g)) = \mathsf s_{1-j}(\phi,g,x^n) 
    \end{equation}
    in any boundary normal coordinate system $\phi$.
\end{proposition}
\begin{proof}
    We proceed by induction. The   case $j=0$ follows from \eqref{eq: r_1} by setting $\mathsf s_1 =- \rho \in \mathcal {SP}_{0,0}^1$. 
    The   case $j=1$ follows from \eqref{eq: r_0} by setting 
    \[ \mathsf s_0 := \frac{1}{2}\rho^{-1}\left(\sum_\gamma D_\gamma\mathsf s_1\partial_\gamma\mathsf s_1 -\mathsf q_1 + \partial_n\mathsf s_1-\mathsf E\mathsf s_1\right),\]
    where 
    \[\mathsf E = - \mathsf G^{-1}\partial_n  \mathsf G\in\mathcal{SP}_{0,1}^0,\]
    and  
    \[\mathsf q_1 = -\eta_\beta \mathsf G^{-1}\mathsf g^{\alpha\beta}\partial_\alpha\mathsf G - \eta_\beta\partial_\alpha\mathsf g^{\alpha\beta}\in\mathcal {SP}_{0,1}^1.\] 
    Moreover,  
by Proposition \ref{prop:thespaceSP}, $\mathsf s_0  \in \mathcal{SP}_{0,1}^0$.

  After constructing the desired $\mathsf s_1,\mathsf s_0,\dots,\mathsf s_{2-j}$, $j\ge 2$,   in view of \eqref{eq: r_{1 - j}} and Proposition \ref{prop:thespaceSP}, 
    \[\mathsf s_{1-j} := \frac{1}{2}\rho ^{-1}\left(\sum_{\substack{0\le k,l\le j-1,\\ |J| = j-k-l}}\frac{1}{J!}D^J\mathsf s_{1-k} \partial^J\mathsf s_{1-l} + \partial_{n}\mathsf s_{2 - j} -\mathsf E\mathsf s_{2 - j}\right) \in \mathcal{SP}_{0,j}^{1-j} \]
     satisfies \eqref{eq: characterization of symbol}.
\end{proof}

The same argument also gives the corresponding statement for integer powers.

\begin{proposition}\label{prop: characterization of symbol A^z}
    For each $z\in\mathbb Z$, there exists $\mathsf s_{z,z-j}\in\mathcal{SP}_{0,j}^{z-j}$, $j\in\mathbb N_0$, such that for any Riemannian manifold $(\Omega^n,g)$ with an $\varepsilon$-collar,   one has
    \[\sigma_{z-j}(\phi,A_t(\Omega,g)^z) = \mathsf s_{z,z-j}(\phi,g)|_{x^n=t} \]
    in any boundary normal coordinate system $\phi$.
\end{proposition}
\begin{proof}
    The case $z=0$ is trivial and the case $z=1$ is given by  Proposition \ref{prop: characterization of symbol}. Recall that if $P_1$ and $P_2$ are two pseudodifferential operators of order $z_1$ and $z_2$, respectively, then the full symbol of $P_1\circ P_2$ is locally given by
    \[\sigma_{z_1+z_2-j}(P_1\circ P_2) = \sum_{\substack{0\le k,l\le j\\|J| = j-k-l}}\frac{1}{J!}D_{\xi}^J\sigma_{z_1-k}(P_1) \partial_x^J \sigma_{z_2-l}(P_2).\]
    For $z>1$, set
    \[\mathsf s_{z,z-j} = \sum_{\substack{0\le k,l\le j\\|J| = j-k-l}}\frac{1}{J!}D^J\mathsf s_{z-1,z-1-k} \partial^J\mathsf s_{1,1-l}.\]
    Then by induction on $z$,  $\mathsf s_{z,z-j}\in\mathcal{SP}_{0,j}^{z-j}$ satisfies the desired condition.

    Finally assume that $z<0$. Since $A^{z}\circ A^{-z}\equiv\Id$ modulo smoothing operators, we may take, in view of the composition formulas above,  
    \[ 
          \mathsf s_{z,z}   =\mathsf s_{-z,-z}^{-1} = (-\rho)^{z}\in\mathcal{SP}_{0,0}^{z},
     \]
     and in general, for any $j \ge 1$,      
     \[     \mathsf s_{z,z-j}   = -\mathsf s_{-z,-z}^{-1}\sum_{\substack{0\le k,l\le j\\(k,l)\neq (j,0)\\ |J| = j-k-l}}\frac{1}{J!}D^J \mathsf s_{z,z-k} \partial^J \mathsf s_{-z,-z-l}\in\mathcal{SP}_{0,j}^{z-j},\quad j\ge 1.
    \]
\end{proof}

\section{The Guillemin--Wodzicki residue of $\Lambda$} \label{sec:GWResidue}

Recall that $\widetilde{\mathcal P}_{w} = \bigoplus_{0\le j\le[w/2]}\mathcal P_{2j,w-2j}$. On the space 
\[\mathcal{SP}_{0,w}^m = \bigoplus_{J} \rho^{m-|J|}\eta^J \mathcal P_{|J|,w-|J|},\] there is an integral map 
\[I:\mathcal{SP}_{0,w}^m\to\widetilde{\mathcal P}_w,\quad \sum_{J}\rho^{m-|J|}\eta^J\mathsf p 
\mapsto \sum_{J} \left(\int_{|\xi|=1}(i\xi)^Jd\xi\right) \mathsf p.\] 
Here we use the fact that  $\int_{|\xi|=1}(i\xi)^J d\xi = 0$ for odd $|J|$. It directly follows from the definition that $I$ maps $\mathcal{SP}_T$ to $\mathcal P_T$ for any type $T$. 

Proposition \ref{prop: characterization of symbol A^z} gives the following structural description of the Guillemin--Wodzicki residue of $A_{x^n}^z$.

\begin{proposition}\label{prop: structure of residue}
    For any integer $z\ge-n+1$, there exists an invariant $\mathsf P_z\in\widetilde{\mathcal P}_{z+n-1}$ such that for any Riemannian manifold $(\Omega^n,g)$ with an $\varepsilon$-collar, one has
    \[\res\left(A_{t}(\Omega,g)^z\right) = \mathsf P_z(g)\big|_{x^n=t}\,  |d\sigma_{t}|.\]
    Here $A_{t}(\Omega,g)^z$ denotes any classical parametrix of $A_{t}(\Omega,g)^{-z}$ for $z<0$.
\end{proposition}
\begin{proof}
    Let $\phi$ be a boundary normal coordinate system. By Proposition \ref{prop: characterization of symbol A^z} and the definition of the Guillemin--Wodzicki residue density,
    \begin{equation*}
        \res\left(A_{t}(\Omega,g)^z\right) = \mathsf Q_z(\phi,g)\big|_{x^n=t}\, |dx'|,
    \end{equation*}
    where $\mathsf Q_z := I(\mathsf s_{z,-n+1})$, and $dx'=dx^1 \cdots dx^{n-1}$ in the coordinate system $\phi$. Since $|d\sigma_t|= \sqrt{\det(( g_t)_{\alpha\beta})}|dx'|$, the metric polynomial $\mathsf P_z = \mathsf G^{-1}\mathsf Q_z$ satisfies
    \begin{equation}
        \label{eq: residue of A^z is P_z}
        \res\left(A_{t}(\Omega,g)^z\right) = \mathsf P_z(\phi,g)\big|_{x^n=t}\, |d\sigma_t|.
    \end{equation}
    
    It remains to show that $\mathsf P_{z}$ is invariant. Since the left-hand side of \eqref{eq: residue of A^z is P_z} is independent of the choice of coordinate system, so is the right-hand side. In particular, this proves the invariance of $\mathsf P_z$ since $|d\sigma_t|$ is independent of the choice of boundary orthonormal basis.  
\end{proof}
\begin{remark}
    \label{rmk: strong invariance}
    In fact, we have shown a stronger invariance of $\mathsf P_z$: the value of $\mathsf P_z(\phi,g)$ at a point is independent of the choice of the  boundary normal coordinate system $\phi$ (which need not be a geodesic normal coordinate system).
\end{remark}

As a first application, we prove the following vanishing result.

\begin{corollary}\label{cor: special manifolds}
    \begin{enumerate}
        \item If $\partial_{x^n}^j(g_{x^n})|_{x^n=0} = 0$ for $1\le j\le n+z-1$, then $\res \Lambda(\Omega,g)^z = 0$ for all  $z\in\mathbb Z$ such that $n+z$ is even.
        \item If the boundary $\Sigma$ has a collar neighborhood that is isometric to $\Sigma\times[0,\varepsilon)$, then $\res \Lambda(\Omega,g)^z = 0$ for all even $z\in\mathbb N_+$. 
    \end{enumerate}
\end{corollary}
\begin{proof}
    For (1), the conclusion is trivial if $z<-n+1$, therefore we  assume $z\ge -n+1$. Suppose $n+z$ is even. Then for a monomial $\mathsf m$ of $\mathsf P_{z}$, its total order 
    \[ \ord^\top(\mathsf m) + \ord^\bot(\mathsf m)\equiv z+n-1 \equiv 1 \mod 2.\] 
    This implies that $\ord^\bot(\mathsf m)$ is nonzero since $\ord^\top(\mathsf m)$ is even for $\mathsf P_{z} \in\widetilde{\mathcal P}_{z+n-1}$. The evaluation of $\mathsf m$ at the boundary of $\Omega$ is zero by the assumption on $g$ and the fact that $A_0\equiv-\Lambda$. This proves (1). For (2), write $z=2k$ with $k\in\mathbb N_+$. By a consequence of \eqref{operatordecomp} (see also \cite[Theorem 2.1]{leeBFKGluingFormula2003}), we have $A_0^2\equiv\Delta^\top$ modulo smoothing operators, and thus $A_0^{2k}\equiv(\Delta^\top)^k$ modulo smoothing operators. But the operator $(\Delta^\top)^k$ is differential, so its Guillemin--Wodzicki residue vanishes. 
\end{proof}
   
Part (2) of Corollary \ref{cor: special manifolds} is an extension of \cite[Corollary 3.7]{polterovichHeatInvariantsSteklov2015}, which states that $\res \Lambda(\Omega,g)^z = 0$ for all even $z$, $2-n\le z\le 0$, in view of \eqref{eq: a_j as residue}.

From now on we fix $n=3$. Consider the invariant $\mathsf P_1$ in Proposition \ref{prop: structure of residue}. 
\begin{proposition}
    There exist $c_1,\dots,c_8\in\mathbb R$ such that
    \begin{equation}\label{eq: form of P(g)} 
	 	\begin{aligned}
	 		\mathsf P_1 = \ & c_1\mathsf H_1^3 + c_2\mathsf H_1\mathsf R^\top + c_3\mathsf H_1\mathsf H_2 + c_4\mathsf {\Delta^\top H_1} + c_5\langle \mathsf h, \mathsf {Rm}^\bot\rangle \\
	 		& + c_6\mathsf H_1\mathsf{Ric}^\bot + c_7\partial_n\mathsf{Ric}^\bot + c_8\partial_n\mathsf R
	 	\end{aligned}
	 \end{equation}
\end{proposition}
\begin{proof} 
According to Proposition \ref{prop:decomforn=3}, there exists $\mathsf U\in\mathcal P_3$ of the form of the right-hand side of \eqref{eq: form of P(g)} such that $\mathsf P_1\equiv\mathsf U\mod\mathcal J$. Since the evaluation of $\mathsf P_1(\phi,g)$ and $\mathsf U(\phi,g)$ at a point is independent of the choice of $\phi$, it follows from Lemma \ref{lem: independent of phi} that $\mathsf P_1 = \mathsf U$.
\end{proof}
It remains to compute these coefficients $c_1,\dots,c_8$.

\section{Proof of Theorem \ref{thm: main thm b1 b2}}\label{sec: proof of thm}

This section is devoted to proving Theorem \ref{thm: main thm b1 b2}. 

\subsection{Computation of coefficients}

We determine the eight coefficients $c_1,\dots,c_8$ in \eqref{eq: form of P(g)} by evaluating $\mathsf P=\mathsf P_1$ on several model metrics.  Throughout the computations below, all geometric quantities and symbol coefficients are evaluated at the boundary $x^n=0$ unless otherwise specified.

\subsubsection{Euclidean balls}
Let $\Omega=\mathbb B^3$ be the unit ball in $\mathbb R^3$ with its Euclidean metric.  Its Steklov eigenvalues are $k\in\mathbb N_0$, with multiplicity $2k+1$, and hence
\[
 \Tr(e^{-t\Lambda})=\sum_{k=0}^\infty(2k+1)e^{-tk}
 =\frac{1+e^{-t}}{(1-e^{-t})^2}
 =\frac2{t^2}+\frac1t+\frac13+O(t).
\]
Thus the coefficient of $t\log t$ vanishes, and therefore
$\Res(A_0(\mathbb B^3,g_{\mathbb B^3}))=0$ by  \eqref{eq: bk-residue-introduction}. On the unit sphere we have 
\[H_1=1, \quad   H_2=1, \quad R^\top=2,\] 
and all ambient curvature terms vanish.  Consequently,
\[
 0=\Res(A_0)=\int_{\mathbb S^2}\mathsf P(g_{\mathbb B^3})\,|d\sigma|
 =4\pi(c_1+2c_2+c_3),
\]
so
\begin{equation}\label{eq: c_1-c_8 1}
 c_1+2c_2+c_3=0.
\end{equation}

\subsubsection{A warped product of a closed surface and an interval}
Let $(\Sigma,g_\Sigma)$ be a closed surface and let $u:[0,1)\to\mathbb R_{>0}$ be smooth.  On
\[
 \Omega=\Sigma\times[0,1),\qquad g_u=u(x^n)^2g_\Sigma+(dx^n)^2,
\]
a direct computation gives
\begin{equation}\label{eq: curvatures of g_u}
\begin{aligned}
 H_1&=-\frac{u'}u,
 &R^\top&=\frac{R^\Sigma}{u^2},\\
 H_2&=\left(\frac{u'}u\right)^2,
 &\Delta^\top H_1&=0,\\
 \langle h,\Rm^\bot\rangle&=\frac{2u'u''}{u^2},
 &H_1\Ric^\bot&=\frac{2u'u''}{u^2},\\
 \partial_n\Ric^\bot&=\frac{2u'u''}{u^2}-\frac{2u'''}u,&
 \partial_nR&=-\frac{2u'}{u^3}R^\Sigma
 +4\left(\frac{u'}u\right)^3-\frac{4u'''}u.
\end{aligned}
\end{equation}
It follows from \eqref{eq: form of P(g)} that
\begin{equation}\label{eq: form of P(g) warped product}
 \mathsf P(g_u)
 =B_1\frac{u'}{u^3}R^\Sigma
 +B_2\left(\frac{u'}u\right)^3
 +B_3\frac{u'u''}{u^2}
 +B_4\frac{u'''}u,
\end{equation}
where
\begin{align}
 B_1&=-c_2-2c_8,\label{eq: c_1-c_8 2}\\
 B_2&=-c_1-c_3+4c_8,\label{eq: c_1-c_8 3}\\
 B_3&=2c_5+2c_6+2c_7,\label{eq: c_1-c_8 4}\\
 B_4&=-2c_7-4c_8.\label{eq: c_1-c_8 5}
\end{align}
Notice in particular that the $u'''$-term comes from
$c_7\partial_n\Ric^\bot+c_8\partial_nR$.

\begin{lemma}\label{lem: B_1, B_2, B_3, B_4}
One has
\[
 B_1=B_2=0,\qquad B_3=-\frac\pi2,\qquad B_4=-\frac\pi4.
\]
\end{lemma}

\begin{proof}
Let $\mathsf p_\bot$ denote the component of a metric polynomial containing no tangential metric jets, and use the analogous notation for symbols.  Choose geodesic normal coordinates $\phi$ for $g_\Sigma$ at a point $q\in\Sigma$, and put $r=|\xi|_{g_\Sigma}$. Denote $s_j := \mathsf s_{1,\bot}(\phi,g)$. The normal part of the Lee--Uhlmann recursion gives
\begin{align*}
 s_1 &= -\frac r u,\\
 s_0 &= -\frac{u'}{2u},\\
 s_{-1} &= -\frac{(u')^2+2uu''}{8ur},\\
 s_{-2} &= -\frac{2u'u''+uu'''}{8r^2}.
\end{align*}

Since the boundary density in these coordinates is $\mathsf G=u^2$, integration over $\{r=1\}$ yields
\begin{equation}\label{eq: P_bot(g_u)}
 \mathsf P_\bot(g_u)
 =-\frac\pi2\frac{u'u''}{u^2}-\frac\pi4\frac{u'''}u.
\end{equation}
The remaining components of $\mathsf P(g_u)$ contain tangential second derivatives of $g_\Sigma$.  By the type decomposition and the two-dimensional curvature identities, the resulting complete contraction is necessarily a constant multiple of $u'R^\Sigma/u^3$.  Comparing with \eqref{eq: form of P(g) warped product} therefore gives
\[
 B_2=0,\qquad B_3=-\frac\pi2,\qquad B_4=-\frac\pi4.
\]

To determine $B_1$, take $\Sigma=\mathbb S^2$ with the unit round metric and $u(x^n)=1-x^n$.  This collar is isometric to a collar of the Euclidean unit ball.  The residue density is rotationally invariant and has zero integral, so it vanishes pointwise.  Since $R^\Sigma=2$, substituting $u(0)=1$, $u'(0)=-1$, and $u''(0)=u'''(0)=0$ into \eqref{eq: form of P(g) warped product} gives
\[
 -2B_1-B_2=0.
\]
As $B_2=0$, we obtain $B_1=0$.
\end{proof}

\subsubsection{A one-directional warped flat torus}

Consider
\[
 \Omega=\mathbb S^1\times\mathbb S^1\times[0,1),\qquad
 g_v=v(x^n)^2g_{\mathbb S^1}+g_{\mathbb S^1}+(dx^n)^2,
\]
where $v>0$ is smooth.  We may assume $v(0)=1$, while the three numbers
$a=v'(0)$, $b=v''(0)$, and $d=v'''(0)$ remain arbitrary.  At the boundary,
\begin{equation}\label{eq: curvatures of g_v}
\begin{aligned}
 H_1&=-\frac{v'}{2v},
 &R^\top&=0,\\
 H_2&=0,
 &\Delta^\top H_1&=0,\\
 \langle h,\Rm^\bot\rangle&=\frac{v'v''}{v^2},
 &H_1\Ric^\bot&=\frac{v'v''}{2v^2},\\
 \partial_n\Ric^\bot&=\frac{v'v''}{v^2}-\frac{v'''}v,&
 \partial_nR&=\frac{2v'v''}{v^2}-\frac{2v'''}v.
\end{aligned}
\end{equation} 
Hence
\[
 \mathsf P(g_v)=C_1\left(\frac{v'}v\right)^3
 +C_2\frac{v'v''}{v^2}+C_3\frac{v'''}v,
\]
with
\begin{align}
 C_1&=-\frac{c_1}{8},\label{eq: c_1-c_8 7}\\
 C_2&=c_5+\frac{c_6}{2}+c_7+2c_8,\label{eq: c_1-c_8 8}\\
 C_3&=-c_7-2c_8.\label{eq: c_1-c_8 9}
\end{align}

\begin{lemma}
One has
\[
 C_1=\frac\pi{64},\qquad C_2=-\frac\pi{32},\qquad C_3=-\frac\pi8.
\]
\end{lemma}
\begin{proof}
Since $\mathsf g_{\alpha\beta,\gamma\delta}(g_v) = \mathsf g_{\alpha\beta,\gamma\delta n}(g_v) =  0$, we have $\mathsf P(g_v)=\mathsf P_\bot(g_v)$. Therefore, by a direct computation using the normal part of the Lee--Uhlmann recursion, at the boundary point, the degree $-2$ symbol is
\[
 r_{-2}=-\frac{\xi_2^2}{8r^8}\Bigl[
 a^3(\xi_2^4-10\xi_1^2\xi_2^2+4\xi_1^4)
 +ab(-2\xi_2^4+5\xi_1^2\xi_2^2+7\xi_1^4)
 +d\,r^4\Bigr],
\]
where $r=(\xi_1^2+\xi_2^2)^{1/2}$.  Using the standard moments on the unit circle gives
\[
 \mathsf P(g_v)=\int_{r=1}r_{-2}\,d\xi
 =\frac\pi{64}(a^3-2ab-8d).
\]
The asserted values of $C_1,C_2,C_3$ follow because $a,b,d$ are arbitrary.  Equivalently, without the normalization $v(0)=1$, the middle term is
$-2v'v''/v^2$.
\end{proof}

\subsubsection{A conformal tangential jet}

Consider 
\[
 \Omega = \mathbb R^2\times [0,1),\quad g_L=e^{F_L}\left((dx^1)^2+(dx^2)^2\right)+(dx^n)^2,
 \quad F_L=L(x^1)^2x^n,
\]
where $L\in\mathbb R$. At the point $q=(0,0,0)$ one has
\begin{equation}\label{eq: curvatures of g_L}
    \Delta^\top H_1=-L,\qquad\partial_n R=-2L, 
\end{equation}
while all the other invariants in \eqref{eq: form of P(g)} vanish.  Indeed, along the boundary, $h_{\alpha\beta}=-(L/2)(x^1)^2\delta_{\alpha\beta}$, while the scalar curvature of the slice metric
$e^{F_L}((dx^1)^2+(dx^2)^2)=e^{2\varphi_L}g_{\mathrm{eucl}}$ is
$-2e^{-2\varphi_L}\Delta\varphi_L$, with $\varphi_L=F_L/2$.
Consequently,
\begin{equation}\label{eq: conformal jet invariant evaluation}
 \mathsf P(g_L)(q)=-L(c_4+2c_8).
\end{equation}

We now compute the same quantity from the symbol recursion.  For notational convenience, set
\[y=x^1,\quad s = x^n, \quad r=(\xi_1^2+\xi_2^2)^{1/2}.\]
Since
\[q_2=e^{-Ly^2s}r^2,\qquad q_1=0,\qquad E=-Ly^2,\]
the terms linear in $L$ in the first three symbol coefficients are
\begin{align*}
 s_1&=-r+\frac L2y^2sr+O(L^2),\\
 s_0&=L\left(\frac{iys\xi_1}{2r}-\frac{y^2}{4}\right)+O(L^2),\\
 s_{-1}&=\frac{Ls}{4r^3}(\xi_2^2-\xi_1^2)+O(L^2).
\end{align*}
At $q$, every omitted term in the recursion for $s_{-2}$ either is quadratic in $L$ or contains a factor $y$ or $s$.  Since
$D_{\xi_1}^2(-r)=\xi_2^2/r^3$, we obtain
\begin{align*}
    s_{-2}(q,\xi)
     =\frac1{2r}\left(
    \frac12D_{\xi_1}^2(-r)\,\partial_y^2s_0
    +\partial_s s_{-1}\right)_{q}+O(L^2) =-\frac{L\xi_1^2}{8r^4}+O(L^2).
   \end{align*}
Therefore
\[
 \mathsf P(g_L)(q)=\int_{r=1}s_{-2}(q,\xi)\,d\xi=-\frac{\pi L}{8}.
\]
Comparing with \eqref{eq: conformal jet invariant evaluation} yields the additional relation
\begin{equation}\label{eq: c_1-c_8 6}
 c_4+2c_8=\frac\pi8.
\end{equation}

\subsubsection{The product of a planar region and a circle}

Let $U\subset\mathbb R^2$ be a bounded smooth planar domain and set
$\Omega=U\times\mathbb S^1$ with the product metric.  If
$\mathbf c$ is an arc-length parametrization of $\partial U$, write $\lambda$ for its signed curvature.  The associated boundary normal coordinates give
\[
 g_\lambda=(1-\lambda(x^1)x^n)^2(dx^1)^2+(dx^2)^2+(dx^n)^2.
\]
At the boundary,
\[
 H_1=\frac\lambda2,\qquad H_2=R^\top=0,
 \qquad \Delta^\top H_1=\frac{\lambda''}{2},
\]
and all ambient curvature terms vanish.  Thus the component of type $(2,1)$ satisfies
\[
 \Proj_{(2,1)}\mathsf P(g_\lambda)=\frac{c_4}{2}\lambda''.
\]
We claim that
\begin{equation}\label{eq: planar circle type computation}
 \Proj_{(2,1)}\mathsf P(g_\lambda)=-\frac\pi{32}\lambda''.
\end{equation}

To compute the universal coefficient of $\lambda''$, it is enough to prescribe at a fixed point
\[
 \lambda(0)=\lambda'(0)=0,\qquad \lambda''(0)=L;
\]
for instance, take $\lambda(x^1)=L(x^1)^2/2$ near the point.  Put
$y=x^1$, $s=x^n$, $r=(\xi_1^2+\xi_2^2)^{1/2}$, and
$\theta=1-Ly^2s/2$.  The terms linear in $L$ in the Lee--Uhlmann recursion are
\begin{align*}
 s_1&=-r-\frac{Ly^2s\xi_1^2}{2r}+O(L^2),\\
 s_0&=L\left(
 \frac{iys\xi_1\xi_2^2}{2r^3}
 +\frac{y^2\xi_2^2}{4r^2}\right)+O(L^2),\\
 s_{-1}&=L\left(
 -\frac{s\xi_1^2\xi_2^2}{2r^5}
 +\frac{iy\xi_1\xi_2^2}{2r^4}\right)+O(L^2).
\end{align*}
Here we used
$q_2=\xi_1^2/\theta^2+\xi_2^2$,
$q_1=-iLys\xi_1+O(L^2)$, and
$E=Ly^2/2+O(L^2)$.
At the chosen point, the recursion for $s_{-2}$ reduces to
\begin{align*}
    s_{-2}(0,\xi)
    &=\frac1{2r}\left(
    D_{\xi_1}(-r)\,\partial_ys_{-1}
    +\frac12D_{\xi_1}^2(-r)\,\partial_y^2s_0
    +\partial_ss_{-1}\right)_{y=s=0}+O(L^2)\\
    &=\frac{L}{8r^6}\bigl(\xi_2^4-4\xi_1^2\xi_2^2\bigr)+O(L^2),
   \end{align*}
where $D_{\xi_1}(-r)=i\xi_1/r$ and $D_{\xi_1}^2(-r)=\xi_2^2/r^3$.
Finally,
\[
 \int_{r=1}\xi_2^4\,d\xi=\frac{3\pi}{4},\qquad
 \int_{r=1}\xi_1^2\xi_2^2\,d\xi=\frac\pi4,
\]
so
\[
 \Proj_{(2,1)}\mathsf P(g_\lambda)
 =\int_{r=1}s_{-2}(0,\xi)\,d\xi
 =-\frac{\pi L}{32}.
\]
This proves \eqref{eq: planar circle type computation}, and comparison with the invariant expression gives
\begin{equation}\label{eq: c_1-c_8 10}
 \frac{c_4}{2}=-\frac\pi{32}.
\end{equation}

\subsection{Proof of  Theorem \ref{thm: main thm b1 b2} and  Corollary \ref{cor: b_1 b_2 constant curvature} }
\begin{proof}[Proof of Theorem \ref{thm: main thm b1 b2}]

Equation \eqref{eq: c_1-c_8 7} gives $c_1=-\pi/8$, while \eqref{eq: c_1-c_8 10} gives $c_4=-\pi/16$.  Combining the latter with \eqref{eq: c_1-c_8 6} yields $c_8=3\pi/32$.  Equations \eqref{eq: c_1-c_8 2}, \eqref{eq: c_1-c_8 3}, and \eqref{eq: c_1-c_8 9} then give
\[
 c_2=-\frac{3\pi}{16},\qquad
 c_3=\frac\pi2,\qquad
 c_7=-\frac\pi{16}.
\]
Finally, \eqref{eq: c_1-c_8 4} and \eqref{eq: c_1-c_8 8} form a nonsingular two-by-two system for $c_5,c_6$, whose solution is
\[
 c_5=-\frac\pi8,\qquad c_6=-\frac\pi{16}.
\]
Thus
\begin{align*}
 c_1&=-\frac\pi8,& c_2&=-\frac{3\pi}{16},& c_3&=\frac\pi2,& c_4&=-\frac\pi{16},\\
 c_5&=-\frac\pi8,& c_6&=-\frac\pi{16},& c_7&=-\frac\pi{16},& c_8&=\frac{3\pi}{32}.
\end{align*}
The remaining relations, namely \eqref{eq: c_1-c_8 1} and \eqref{eq: c_1-c_8 5}, are then automatically satisfied and provide consistency checks on the model computations.  Substitution into \eqref{eq: form of P(g)} gives the explicit formula for $\mathsf P_1$.

Recall that
\[
 \res(A_t(\Omega,g)^2)=\mathsf P_2(g)\big|_{x^n=t}|d\sigma_t|,
 \qquad \mathsf P_j=\mathsf G^{-1}\mathsf Q_j,
 \qquad \mathsf Q_j=I(\mathsf s_{j,-2}).
\]
Integrating \eqref{eq: sigma A^2 and sigma A} over $\{|\xi|=1\}$ gives
\[
 \mathsf Q_2=\mathsf E\mathsf Q_1-\partial_n\mathsf Q_1.
\]
Since $\partial_n\mathsf G=-\mathsf E\mathsf G$, we obtain
\begin{equation}\label{eq: P2 = 2EP1 - dn P1}
 \mathsf P_2
 =\mathsf G^{-1}(\mathsf E\mathsf Q_1-\partial_n\mathsf Q_1)
 =2\mathsf E\mathsf P_1-\partial_n\mathsf P_1.
\end{equation}
Here $\mathsf E=2\mathsf H_1$.  It remains to differentiate the eight invariant terms in \eqref{eq: form of P(g)}.  The required identities are

\begin{align*}
 \partial_n(\mathsf H_1^3)
 &=6\mathsf H_1^4-3\mathsf H_1^2\mathsf H_2
   +\frac32\mathsf H_1^2\mathsf{Ric}^\bot,\\[2mm]
 \partial_n(\mathsf H_1\mathsf R^\top)
 &=2\mathsf H_1^2(\mathsf R^\top+2\mathsf{Ric}^\bot+2\mathsf H_2)
   -\mathsf H_2\mathsf R^\top+\frac12\mathsf{Ric}^\bot\mathsf R^\top\\
 &\quad+\mathsf H_1\partial_n\mathsf R
   -2\mathsf H_1\partial_n\mathsf{Ric}^\bot
   -2\mathsf H_1\langle\mathsf h,\mathsf{Rm}^\bot\rangle,\\[2mm]
 \partial_n(\mathsf H_1\mathsf H_2)
 &=4\mathsf H_1^2\mathsf H_2-\mathsf H_2^2
   +\frac12\mathsf H_2\mathsf{Ric}^\bot
   +2\mathsf H_1^2\mathsf{Ric}^\bot
   -\mathsf H_1\langle\mathsf h,\mathsf{Rm}^\bot\rangle,\\[2mm]
 \partial_n(\mathsf\Delta^\top\mathsf H_1)
 &=2\langle\mathsf h,\operatorname{Hess}^\top\mathsf H_1\rangle
   +4\mathsf H_1\mathsf\Delta^\top\mathsf H_1
   +2|\nabla^\top\mathsf H_1|^2\\
 &\quad+2\langle\dvg^\top\mathsf h,\nabla^\top\mathsf H_1\rangle
   -\mathsf\Delta^\top\mathsf H_2
   +\frac12\mathsf\Delta^\top\mathsf{Ric}^\bot,\\[2mm]
 \partial_n\langle\mathsf h,\mathsf{Rm}^\bot\rangle
 &=\langle\mathsf h^2,\mathsf{Rm}^\bot\rangle
   +\langle\mathsf{Rm}^\bot,\mathsf{Rm}^\bot\rangle
   +\langle\mathsf h,\mathsf{Rm}_{;n}^\bot\rangle,\\[2mm]
 \partial_n(\mathsf H_1\mathsf{Ric}^\bot)
 &=2\mathsf H_1^2\mathsf{Ric}^\bot
   -\mathsf H_2\mathsf{Ric}^\bot
   +\frac12(\mathsf{Ric}^\bot)^2
   +\mathsf H_1\partial_n\mathsf{Ric}^\bot,\\[2mm]
 \partial_n(\partial_n\mathsf{Ric}^\bot)
 &=\partial_n^2\mathsf{Ric}^\bot,\\
 \partial_n(\partial_n\mathsf R)
 &=\partial_n^2\mathsf R.
\end{align*}

See Section \ref{subsec: exp of inv poly} for the definition of the terms appearing above and Appendix \ref{apd: computation} for proofs. Substituting the values of $c_1,\dots,c_8$ and the preceding normal-derivative identities into \eqref{eq: P2 = 2EP1 - dn P1}, and then collecting equal complete contractions, gives
\begin{equation}
\label{eq: general P_2}
\begin{aligned}
\mathsf P_2 = &\ \frac{\pi}{4}\mathsf H_1^4 + \frac{3\pi}{8}\mathsf H_1^2\mathsf H_2 - \frac{3\pi}{8}\mathsf H_1^2\mathsf R^\top - \frac{3\pi}{16}\mathsf H_1^2\mathsf{Ric}^\bot - \frac{3\pi}{16}\mathsf H_2\mathsf R^\top \\
&\ + \frac{3\pi}{32}\mathsf{Ric}^\bot\mathsf R^\top + \frac{\pi}{2}\mathsf H_2^2 - \frac{5\pi}{16}\mathsf H_2\mathsf{Ric}^\bot + \frac{\pi}{32}(\mathsf{Ric}^\bot)^2 - \frac{3\pi}{8}\mathsf H_1\langle\mathsf h,\mathsf{Rm}^\bot\rangle\\
&\ - \frac{9\pi}{16}\mathsf H_1\partial_n\mathsf{Ric}^\bot + \frac{9\pi}{16}\mathsf H_1\partial_n\mathsf R + \frac{\pi}{8}\langle\mathsf h,\Hess^\top\mathsf H_1\rangle + \frac{\pi}{8}|\nabla^\top\mathsf H_1|^2\\
&\ + \frac{\pi}{8}\langle\dvg^\top\mathsf h,\nabla^\top\mathsf H_1\rangle - \frac{\pi}{16}\Delta^\top\mathsf H_2 + \frac{\pi}{32}\Delta^\top\mathsf{Ric}^\bot + \frac{\pi}{8}\langle\mathsf h^2,\mathsf{Rm}^\bot\rangle\\
&\ + \frac{\pi}{8}\langle\mathsf{Rm}^\bot,\mathsf{Rm}^\bot\rangle + \frac{\pi}{8}\langle\mathsf h,\mathsf{Rm}^\bot_{;n}\rangle + \frac{\pi}{16}\partial_n^2\mathsf{Ric}^\bot - \frac{3\pi}{32}\partial_n^2\mathsf R.
\end{aligned}
\end{equation}
Then Theorem \ref{thm: main thm b1 b2} follows from \eqref{eq: bk-residue-introduction} and the fact that $\Lambda(\Omega,g)\equiv -A_0(\Omega,g)$ modulo smoothing operators. 
\end{proof}

\begin{proof}[Proof of Corollary \ref{cor: b_1 b_2 constant curvature}]
Assume now that $\Omega$ has constant sectional curvature $\kappa$.
Then
\[\Ric^\bot=2\kappa, \qquad\Rm^\bot=\kappa g, \qquad R^\top=2\kappa+2H_2,\]
and all covariant derivatives of the ambient curvature vanish. Moreover,
\[\langle h,\Rm^\bot\rangle=2\kappa  H_1,\]
\[\langle  h^2,\Rm^\bot\rangle=\kappa(4 H_1^2-2  H_2), \qquad\langle\Rm^\bot,\Rm^\bot\rangle=2\kappa^2.\]
The Codazzi equation gives
\[\dvg^\top h=2\nabla^\top H_1.\]

Substitution into \eqref{eq: general P_2} yields the pointwise identities
\[\widehat b_1 = \frac{\pi}{8} H_1^3 - \frac{\pi}{8}H_1H_2 + \frac{\pi}{16}\Delta^\top H_1 + \frac{3\pi}{4}H_1\kappa\]
and
\begin{align*}
-\widehat b_2 = &\
\frac{\pi}{4} H_1^4 - \frac{3\pi}{8} H_1^2  H_2 + \frac{\pi}{8} H_2^2 - \frac{11\pi}{8}\kappa H_1^2 - \frac{7\pi}{8}\kappa H_2 + \frac{3\pi}{4}\kappa^2 \\
&\ +\frac{\pi}{8}\langle h,\Hess^\top H_1\rangle + \frac{3\pi}{8}|\nabla^\top H_1|^2 - \frac{\pi}{16}\Delta^\top H_2.
\end{align*} 

Note that
\[\int_\Sigma\langle h,\operatorname{Hess}^\top H_1\rangle = -2\int_\Sigma|\nabla^\top H_1|^2,\]
and $\int_\Sigma\Delta^\top H_2=0$.  
Finally, the Gauss equation and the Gauss--Bonnet theorem give
\[\int_\Sigma H_2\,|d\sigma_g|=2\pi\chi(\Sigma)-\kappa\Area(\Sigma,g).\] 
The integral formulas \eqref{eq: b1} and \eqref{eq: b2} follow.
\end{proof}

\section{Applications to spectral geometry}\label{sec:apptoSpecGeom}

\subsection{Steklov spectrum and the boundary topology}

Let $(\Omega,g)$ be a compact $3$-dimensional Riemannian manifold with constant curvature $\kappa$ and smooth boundary $\Sigma$.  

Let $\kappa \in \mathbb R$. As usual we denote 
\begin{align*}
    \sn_\kappa(t) & =\begin{cases}
        \frac{\sin(\sqrt{\kappa}t)}{\sqrt{\kappa}}, &   \text{if } \kappa > 0,\\
        t, &   \text{if }  \kappa = 0,\\
        \frac{\sinh(\sqrt{-\kappa}t)}{\sqrt{-\kappa}}, &  \text{if }  \kappa < 0,
    \end{cases} \\ 
    \cn_\kappa(t) & = \sn_\kappa'(t),\\
    \tn_\kappa(t) & = \sn_\kappa(t)/\cn_\kappa(t).
\end{align*}
Let $\mathbb B_\kappa(\rho)$ denote a geodesic ball of radius $\rho$ in the simply connected space form $\mathbb M_\kappa$ with constant curvature $\kappa$. Assume $\rho < \frac{\pi}{\sqrt{\kappa}}$ if $\kappa > 0$. The boundary $\partial\mathbb B_\kappa(\rho)$ is smooth and has scalar curvature $2/\sn_\kappa^2(\rho)$, mean curvature $1/\tn_\kappa(\rho)$ and area $4\pi \sn_{\kappa}^2(\rho)$. These quantities distinguish the model geodesic balls.
\begin{proposition}
    If $\Spec(\Lambda(\mathbb B_{\kappa_1}(\rho_1))) = \Spec(\Lambda(\mathbb B_{\kappa_2}(\rho_2)))$, then $\kappa_1 = \kappa_2$ and $\rho_1 = \rho_2$.
\end{proposition}
\begin{proof}
In view of \eqref{eq: a0} and \eqref{eq: a1}, we have
\[\sn_{\kappa_1}(\rho_1) = \sn_{\kappa_2}(\rho_2),\quad \cn_{\kappa_1}(\rho_1) = \cn_{\kappa_2}(\rho_2).\]
Hence 
\[\int_{\partial\mathbb B_{\kappa_1}(\rho_1)} H_1^2 =  \int_{\partial\mathbb B_{\kappa_2}(\rho_2)} H_1^2.\]
Combining this with the spectral invariant \eqref{eq: a2}, we conclude that $\kappa_1 = \kappa_2$. Therefore, $\sn_{\kappa_1}(\rho_1) = \sn_{\kappa_2}(\rho_2)$,  $\cn_{\kappa_1}(\rho_1) = \cn_{\kappa_2}(\rho_2)$ and hence, $\rho_1 = \rho_2$.
\end{proof}

From now on, fix $\kappa\in\mathbb R$. Here and below, we suppress the area density in boundary integrals when it is clear from the context. Linear combinations of \eqref{eq: a0}--\eqref{eq: a2} and \eqref{eq: b1}--\eqref{eq: b2} show that the following quantities are spectral invariants: 
\begin{align*}
    & \tilde a_0(\Omega,g) := \Area(\Sigma,g),\\
    & \tilde a_1(\Omega,g) := \int_\Sigma H_1,\\
    & \tilde a_2(\Omega,g) := \int_\Sigma H_1^2 + \frac{2\pi}{3}\chi(\Sigma),\\
    & \tilde b_1(\Omega,g) := \int_\Sigma H_1(H_1^2-H_2),\\
    & \tilde b_2(\Omega,g) := \int_\Sigma\left((2H_1^2-H_2)(H_1^2-H_2) + |\nabla^\top H_1|^2\right) - \frac{20\kappa\pi}{3}\chi(\Sigma).
\end{align*}

Recall that an isometrically immersed surface $i:(\Sigma,g_\Sigma)\to(M^3,g_M)$ is called {\em totally umbilical} if its second fundamental form is proportional to $g_\Sigma$ everywhere, or equivalently, the principal curvatures are equal at every point. A classical result in differential geometry states that every immersed connected totally umbilical surface in a constant curvature space is either totally geodesic or has nonzero constant mean curvature \cite[Chapter 7.D]{spivak1970comprehensive}. In the latter case, by the Gauss equation, it also has constant scalar curvature. In particular, if $M$ is a simply connected space form and $\Sigma$ is closed, then $\Sigma$ must be a geodesic sphere.

Note that for any immersed surface $\Sigma$ in $M^3$, we have
\begin{equation}\label{eq: H1^2 ge H2}
    H_1^2 = \left(\frac{\lambda_1 + \lambda_2}{2}\right)^2\ge \lambda_1\lambda_2 = H_2.
\end{equation} 
Thus $\Sigma$ is totally umbilical if and only if 
 \[H_1^2 - H_2 = 0.\]

For simplicity, set
\[X_1 = \tilde a_2 - \frac{2\pi}{3}\chi(\Sigma) - \frac{\tilde a_1^2}{\tilde a_0},\quad X_2 = \tilde b_1 - \frac{\tilde a_1}{\tilde a_0}\left(\tilde a_2 + \kappa\tilde a_0 - \frac{8\pi}{3}\chi(\Sigma)\right),\]
and
\[X_3 = \tilde b_2 + \frac{20\kappa\pi}{3}\chi(\Sigma) - \frac{1}{\tilde a_0}\left(\tilde a_2 + \kappa\tilde a_0 - \frac{8\pi}{3}\chi(\Sigma)\right)^2.\]
\begin{lemma}
    \label{lem: spectrum controls chi} 
    The matrix 
    \begin{equation}
        \label{eq: constraint for chi}
        \begin{pmatrix}
            X_1 & X_2\\
            X_2 & X_3
        \end{pmatrix}
    \end{equation}
    is positive semidefinite. Moreover, it vanishes if and only if  $\Sigma$ is one of the following:
    \begin{enumerate}[(1)]
        \item a disjoint union of totally umbilical surfaces  with principal curvature $\frac{\tilde a_1}{\tilde a_0}$;
        \item a disjoint union of  minimal surfaces with $H_2 = \frac{2\pi\chi(\Sigma)}{\tilde a_0} - \kappa$.
    \end{enumerate}
\end{lemma}

\begin{proof}
    The key ingredient is the following consequence of the Cauchy--Schwarz inequality: for any smooth function $f$, we have
    \[\int_\Sigma(H_1^2-H_2+f)^2\ge \frac{1}{\Area(\Sigma,g)}\left(\int_\Sigma (H_1^2-H_2 + f) \right)^2.\]
    In particular, setting $f = tH_1$ for $t\in\mathbb R$, the left-hand side expands to
    \[
        \int_\Sigma(H_1^2-H_2)^2 + 2t\int_\Sigma H_1(H_1^2-H_2)  + t^2\int_\Sigma H_1^2
    \]
    which, by the definitions of $\tilde a_2, \tilde b_1$ and $\tilde b_2$, is bounded from above by
    \[
        \tilde b_2 + \frac{20\kappa\pi}{3}\chi(\Sigma) + 2t\tilde b_1 + t^2\left(\tilde a_2-\frac{2\pi}{3}\chi(\Sigma)\right).
    \]
    Meanwhile, the right-hand side  can be written as 
    \begin{align*}
        & \frac{1}{\Area(\Sigma,g)}\left\{\left(\int_\Sigma H_1^2-H_2 + 2tH_1\right)\int_\Sigma\left(H_1^2-H_2\right) + t^2\left(\int_\Sigma H_1\right)^2\right\}\\
        = & \frac{1}{\tilde a_0}\left\{\left(\tilde a_2 + \kappa\tilde a_0 - \frac{8\pi}{3}\chi(\Sigma) + 2t\tilde a_1\right)\left(\tilde a_2 + \kappa\tilde a_0 - \frac{8\pi}{3}\chi(\Sigma)\right) + t^2\tilde a_1^2\right\},
    \end{align*}
    where we used the Gauss equation $K^\Sigma - \kappa = H_2$ and the Gauss--Bonnet theorem. Therefore, we have
    \[X_1t^2 + 2X_2t + X_3\ge 0\]
    for all $t\in\mathbb R$. The validity of this inequality for all real $t$ is precisely equivalent to the positive semidefiniteness of the matrix \eqref{eq: constraint for chi}.
    
    Next, we analyze the case where the matrix \eqref{eq: constraint for chi} vanishes, which occurs if and only if $X_1t^2+2X_2t+X_3=0$ for all $t\in\mathbb R$, meaning that all the aforementioned inequalities must hold with equality. That is,
    \[\begin{cases}
    	\int_\Sigma(H_1^2-H_2+tH_1)^2 \, d\mu = \frac{1}{\Area(\Sigma,g)}\left(\int_\Sigma (H_1^2-H_2 + tH_1) \, d\mu \right)^2, & \forall t\in\mathbb R,\\
    	\int_\Sigma |\nabla^\top H_1|^2 \, d\mu = 0,\\
    	\int_{\Sigma}H_1^2(H_1^2 - H_2) \, d\mu = 0.
    \end{cases}\] 
    
    By the equality case of the Cauchy--Schwarz inequality, the first relation implies that $H_1^2 - H_2 + tH_1$ must be a constant function for every $t \in \mathbb{R}$. As a result, both $H_1$ and $H_1^2 - H_2$ are constant on $\Sigma$. This is consistent with the second equality, which independently requires $H_1$ to be locally constant. Finally, the third equality implies that on each connected component of $\Sigma$, either $H_1 \equiv 0$ or $H_1^2 - H_2 \equiv 0$. 
    Combined with the fact that both $H_1$ and $H_1^2 - H_2$ are  constant on the whole of $\Sigma$, we see that either $H_1 \equiv 0$ on $\Sigma$, or $H_1^2 - H_2 \equiv 0$ on $\Sigma$. 
     
   If $H_1\equiv0$ on $\Sigma$, then the Gauss equation and Gauss--Bonnet theorem imply
   \[
   H_2=\frac{1}{\tilde a_0}\int_\Sigma H_2\,d\mu
   =\frac{2\pi\chi(\Sigma)}{\tilde a_0}-\kappa,
   \]
   which is the second alternative. If instead $H_1^2-H_2\equiv0$, then $\Sigma$ is totally umbilical by \eqref{eq: H1^2 ge H2}; moreover, the constancy of $H_1$ and the definition of $\tilde a_1$ give
   \[
   H_1=\frac{1}{\tilde a_0}\int_\Sigma H_1\,d\mu=\frac{\tilde a_1}{\tilde a_0}.
   \]
   Hence every connected component has both principal curvatures equal to $\tilde a_1/\tilde a_0$, which is the first alternative. 
   
   Conversely, in either case the function $H_1^2 - H_2 + tH_1$ is a constant depending on $t$, and $|\nabla^\top H_1|^2 = H_1^2(H_1^2 - H_2) = 0$. Therefore, the preceding inequalities are equalities, and hence $X_1 = X_2 = X_3=0$.
\end{proof}
      
\begin{lemma}
    \label{lem: spectrum controls chi sphere}
    If $(\Omega,g)$ is Steklov isospectral to $\mathbb B_\kappa(\rho)$, then $\chi(\Sigma) = 2$. Moreover, $\Sigma$ is totally umbilical with principal curvatures equal to $1/\tn_\kappa(\rho)$, and is isometric to either $\partial\mathbb B_\kappa(\rho)$ or the disjoint union of two copies of $\partial\mathbb B_\kappa(\rho)/\pm$. 
\end{lemma}
\begin{proof}
    By assumption, we have 
    \begin{equation}
    \label{eq: H2 is nonnegative}
        \tilde a_0 = 4\pi \sn_\kappa^2(\rho),\quad \tilde a_1 = 4\pi \sn_\kappa(\rho)\cn_\kappa(\rho),\quad \tilde a_2 = 4\pi \cn_\kappa^2(\rho) + \frac{4\pi}{3},
    \end{equation}
    and
    \[\tilde b_1 = 0,\quad \tilde b_2 = -\frac{40\kappa\pi}{3}.\]
    Substituting these formulas into Lemma \ref{lem: spectrum controls chi} and using $\cn_\kappa^2 + \kappa \sn_\kappa^2 = 1$, we obtain the following constraint on $\chi(\Sigma)$:
    \begin{equation}
        \label{eq: constraint on chi sphere}
        \begin{pmatrix}
            \frac{2\pi}{3}(2-\chi(\Sigma)) & -\frac{8\pi}{3\tn_\kappa(\rho)}\left(2-\chi(\Sigma)\right)\\
            -\frac{8\pi}{3\tn_\kappa(\rho)}\left(2-\chi(\Sigma)\right) & -\frac{4\pi}{9\sn_\kappa^2(\rho)}(2-\chi(\Sigma))(15\kappa \sn_\kappa^2(\rho) + 4(2-\chi(\Sigma)))
        \end{pmatrix}
        \ge 0. 
    \end{equation}
    Or, equivalently, by the standard criterion for positive semidefiniteness, $\chi(\Sigma)$ satisfies the following inequalities:
    \[\frac{2\pi}{3}(2-\chi(\Sigma))\ge 0\quad\text{and}\quad -\frac{8\pi^2}{27\sn_\kappa^2(\rho)}(2-\chi(\Sigma))^2(4(2-\chi(\Sigma)) + 15 + 9\cn_\kappa^2(\rho))\ge0.\]
    Therefore, $\chi(\Sigma) = 2$. 
    The matrix in \eqref{eq: constraint on chi sphere} is then zero. If $\Sigma$ is not totally umbilical, by Lemma \ref{lem: spectrum controls chi}, $\Sigma$ is minimal and has constant $H_2 = 1/\tn_\kappa^2(\rho)\ge 0$. It follows that $H_1^2 - H_2\le 0$ and hence $H_1^2 - H_2$ is identically zero, which contradicts the assumption that $\Sigma$ is not totally umbilical. Therefore, $\Sigma$ is totally umbilical with principal curvatures equal to $1/\tn_\kappa(\rho)$ and scalar curvature $2/\sn^2_\kappa(\rho)$ by the Gauss equation.  
    By the Killing--Hopf theorem, $\Sigma$ is isometric to either $\partial\mathbb B_\kappa(\rho)$ or the disjoint union of two copies of $\partial\mathbb B_\kappa(\rho)/\pm$.
\end{proof}

\subsection{Euclidean sphere shells}
When $\Omega\subset\mathbb R^3$ is a smooth domain, the spectral invariants $\tilde a_2$ and $\tilde b_2$ yield the following rigidity statement:
\begin{proposition}
    \label{prop: isospectral to totally umbilical}
    Let $\Omega^*\subset\mathbb R^3$ be a smooth domain whose boundary components are round spheres. If $\Omega\subset\mathbb R^3$ is Steklov isospectral to $\Omega^*$, then the boundary components of $\Omega$ are round spheres and $\chi(\partial\Omega) = \chi(\partial\Omega^*)$. 
\end{proposition}
\begin{proof}
    By isospectrality, $\tilde b_2(\Omega)$ vanishes since $H_1^2-H_2$ and $\nabla^\top H_1$ vanish on Euclidean spheres. It follows that $\partial\Omega$ is also totally umbilical, which means that each component of  $\partial\Omega$ must be a round sphere by \cite[Theorem 19]{spivak1970comprehensive}. As a result, we have 
    \[\int_{\partial\Omega}H_1^2 = \int_{\partial\Omega}H_2 = 2\pi\chi(\partial\Omega)\]
    by the Gauss--Bonnet theorem. Therefore, it follows from $\tilde a_2(\Omega) = \tilde a_2(\Omega^*)$ that $\chi(\partial\Omega) = \chi(\partial\Omega^*)$.
\end{proof}

In particular, taking $\Omega^*=A_{r,R}$ reduces Theorem~\ref{thm:concentric_spherical_shell_inverse} to determining the radii and relative positions of the two spherical boundary components.  
\begin{proof}[Proof of Theorem \ref{thm:concentric_spherical_shell_inverse}]
    First, by Proposition \ref{prop: isospectral to totally umbilical}, $\Omega =  B(p^+,\rho_+)\setminus \overline{B(p^-,\rho_-)}$, where $B(p^\pm,\rho_\pm)$ denotes the open ball of radius $\rho_\pm$ centered at $p^\pm\in\mathbb R^3$. The outer boundary sphere has
    \[H_1=\frac{1}{\rho_+},\]
    whereas on the inner boundary sphere the outward unit normal of $\Omega$ points into the hole, so
    \[H_1=-\frac{1}{\rho_-}.\]
    Hence
    \[\tilde a_0(\Omega) = 4\pi(\rho_+^2+\rho_-^2), \quad \tilde a_1(\Omega) = 4\pi(\rho_+-\rho_-).\]
    For the concentric shell $A_{r,R}$,
\[\widetilde a_0(A_{r,R}) = 4\pi(R^2+r^2), \quad
\widetilde a_1(A_{r,R}) = 4\pi(R-r).\]
Thus isospectrality gives $R = \rho^+$ and $r = \rho^-$.

Finally, using a result of I. Ftouhi \cite[Theorem 1.1]{ftouhiWherePlaceSpherical2025}, the first nonzero Steklov eigenvalue of $\Omega$ achieves its maximum if and only if $\Omega$ is isometric to $A_{r,R}$. This completes the proof.
\end{proof}

\section{Proof of Theorem \ref{thm: determine 3d geodesic balls}}
\label{sec:ProofThm1}

In this section we apply Corollary \ref{cor: b_1 b_2 constant curvature} to prove Theorem \ref{thm: determine 3d geodesic balls}. The proof can be divided into two steps. First, we will show that the spectral invariants force the boundary to have the same Euler characteristic and the same principal curvatures as the boundary of the geodesic ball. As a consequence, we will prove
\begin{proposition}\label{prop: spectrum determines boundary}
	Let $(\Omega,g)$ be a compact $3$-dimensional Riemannian manifold with constant curvature $\kappa$ and a smooth boundary $\Sigma$. 
    If $\Spec(\Lambda(\Omega,g)) = \Spec(\Lambda(\mathbb B_\kappa(\rho)))$ for some $\rho > 0$,  then $\Sigma$ is totally umbilical with principal curvatures $1/\tn_\kappa(\rho)$, and is isometric to the geodesic sphere $\partial\mathbb B_\kappa(\rho)$.
\end{proposition}

Second, by carefully analyzing the interior topology, we can prove the following Riemannian geometry result (in any dimension) which may be of independent interest:

\begin{proposition}\label{prop:iso-to-geod-ball}
	Let $(\Omega,g)$ be a compact connected  Riemannian manifold with constant sectional curvature $\kappa$ and a smooth boundary $\Sigma$. 	In the case $\kappa > 0$ we further assume $2\sqrt{\kappa}\rho \le \pi$.  If $\Sigma$ is totally umbilical and is 
	isometric to $\partial\mathbb B_\kappa(\rho)$, and if the signs of principal curvatures of $\partial\Omega$ agree with those of $\partial\mathbb B_\kappa(\rho)$  
    then $\Omega$ is isometric to the geodesic ball  $\mathbb B_\kappa(\rho)$.  
\end{proposition} 

Theorem \ref{thm: determine 3d geodesic balls} follows immediately from these two propositions.

\subsection{Proof of Proposition \ref{prop:iso-to-geod-ball}}

More generally, let $(\Omega,g)$ be a compact Riemannian manifold of dimension $n$ with constant curvature $\kappa$ and smooth boundary $\Sigma$.  

\begin{lemma}\label{lem: isometric collar}
    Let $(\Omega_j,g_j)$, $j=1,2$,  be compact Riemannian manifolds of constant sectional curvature $\kappa$ and smooth boundaries $\Sigma_j$. Denote the second fundamental form of $\Sigma_j$ by $h_j$. If there exists an isometry $f:\Sigma_1\to\Sigma_2$ such that $f^*h_2 = h_1$, then $f$ can be extended to an isometry from a collar neighborhood of $\Sigma_1$ to a collar neighborhood of $\Sigma_2$.
\end{lemma}
\begin{proof}
    Let $(x_j^1,\dots,x_j^{n-1},x_j^n)$ be a boundary normal coordinate system defined on $U_j\subset\Omega_j$ such that $x_1^\alpha = x_2^\alpha \circ f$.  
    Let $g_j(t)$ denote the matrix representation of the pull-back Riemannian metric $(g_j)_t$ with respect to the coordinates $(x_j^1,\dots,x_j^{n-1})$, and let $h_j(t) = -\frac{1}{2}\partial_tg_j(t)$. By definition and the formula \eqref{eq: g (0,2)}, $g_j(t)$ and $h_j(t)$ satisfy the following system of first-order differential equations:
    \begin{equation}
        \label{eq: ODE near boundary}
        \begin{cases}
            \dot g_j(t) = -2h_j(t),\\
            \dot h_j(t) = -h_j(t)g_j(t)^{-1}h_j(t) + \kappa g_j(t).
        \end{cases}
    \end{equation}
    For each $j$, equation \eqref{eq: ODE near boundary} has a unique small-time solution $(g_j(t),h_j(t))$ with initial data $g_j(0) = g_j|_{\Sigma_j}$ and $h_j(0) = h_j$. However, $(\tilde g_1(t),\tilde h_1(t)) = (f^*g_2(t),f^* h_2(t))$ is also a solution of \eqref{eq: ODE near boundary} with initial values $\tilde g_1(0) = f^*(g_2|_{\Sigma_2}) = g_1|_{\Sigma_1}$ and $\tilde h_1(0) = f^*(h_2) = h_1$. Therefore, $\tilde g_1(t) = g_1(t)$ and thus, a sufficiently small neighborhood of $U_1$ is isometric to a neighborhood of $U_2$ via the coordinate identification   $(x_1^1,\dots,x_1^{n-1},x_1^n)\mapsto (x_2^1,\dots,x_2^{n-1},x_2^n)$. By taking finitely many coordinate charts covering $\Sigma_1$, we can extend $f$ to an isometry from a collar neighborhood of $\Sigma_1$ to a collar neighborhood of $\Sigma_2$.
\end{proof}
\begin{lemma}\label{lem: embed into space form} 
    Suppose $\Sigma = \bigsqcup_{j=1}^N\Sigma_j$. If each $\Sigma_j$ is totally umbilical and is isometric to $\partial\mathbb B_\kappa(\rho_j)$ for some $\rho_j>0$, then $\Omega$ can be embedded isometrically into a space form $M_\kappa$ of the same dimension with constant curvature $\kappa$. Moreover, $M_\kappa$ is compact if $\kappa > 0$.
\end{lemma}

\begin{proof}
    Since $\Sigma$ is totally umbilical, the second fundamental form of each $\Sigma_j$ is a constant multiple of the metric tensor. If we properly choose the direction of the unit normal along $\partial\mathbb B_\kappa(\rho)$, the isometry $\Sigma_j\to\partial\mathbb B_\kappa(\rho_j)$ then preserves the second fundamental form by the Gauss equation. By Lemma \ref{lem: isometric collar}, each $\Sigma_j$ has a collar neighborhood that is isometric to a collar neighborhood of the boundary of a geodesic ball $B_j\subset\mathbb M_\kappa$ of radius $\rho_j$. Therefore, we can form a Riemannian manifold of constant sectional curvature $\kappa$ by attaching to each $\Sigma_j\subset\Omega$ the region $\Omega_j := \mathbb M_\kappa\setminus B_j$. Denote the resulting manifold by $M_\kappa$. To prove $M_\kappa$ is complete, let $F$ be a bounded closed subset of $M_\kappa$. The set $F\cap\Omega$ is compact since $\Omega$ is compact, and the set $F\cap\Omega_j$ is compact since it can be identified with a bounded closed subset of $\mathbb M_\kappa$. Because $F$ is arbitrary, it follows from Hopf--Rinow theorem that $M_\kappa$ is complete.
\end{proof}

\begin{lemma}\label{lem: disjoint balls}
    Let $B_1$ and $B_2$ be geodesic balls of the same radius $\rho$ in $\mathbb M_\kappa$ centered at $o_1$ and $o_2$, respectively. If the geodesic spheres $\partial B_1$ and $\partial B_2$ are disjoint, and if $2\sqrt{\kappa}\rho \le \pi$ when $\kappa>0$, then $\overline B_1$ and $\overline B_2$ are disjoint.
\end{lemma}

\begin{proof}
	First we show, under the stated assumption, that $d(o_1,o_2)>2\rho$. Otherwise, let $p$ be the midpoint of a minimizing geodesic from $o_1$ to $o_2$ and set
	\[
	a:=d(o_1,p)=d(p,o_2)\le\rho.
	\]
	Let $\ell$ be a geodesic through $p$ perpendicular to this minimizing geodesic, and choose $q\in\ell$ such that $b:=d(p,q)$ is given by 
	\[
	b=\sqrt{\rho^2-a^2}\quad\text{if }\kappa=0,
	\qquad
	b=\cn_\kappa^{-1}\!\left(\frac{\cn_\kappa(\rho)}{\cn_\kappa(a)}\right)
	\quad\text{if }\kappa\ne0.
	\]
	The law of cosines gives $d(o_1,q)=d(o_2,q)=\rho$, contradicting $\partial B_1\cap\partial B_2=\varnothing$. When $\kappa>0$, the assumption $2\sqrt{\kappa}\rho\le\pi$ ensures that the inverse above is well defined, because $0\le\cn_\kappa(\rho)\le\cn_\kappa(a)$ and $\cn_\kappa$ is decreasing on $[0,\frac{\pi}{2\sqrt{\kappa}}]$.
	
	Now we prove the lemma. If $p\in \overline B_1\cap \overline B_2$, then 
	\[d(o_1,o_2) \le d(o_1,p) + d(o_2,p) \le \rho + \rho = 2\rho,\]
	which is a contradiction. Therefore, $\overline B_1\cap\overline B_2=\varnothing$.
\end{proof} 

\begin{lemma}
    \label{lem: local diffeomorphism implies diffeomorphism}
    Let $\Omega_1$ and $\Omega_2$ be two compact smooth manifolds with boundary. Assume that $\Omega_2$ is connected. If $f:\Omega_1\to\Omega_2$ is a local diffeomorphism that maps $\partial\Omega_1$ to $\partial\Omega_2$ as a smooth embedding, then $f$ is a diffeomorphism.
\end{lemma}
\begin{proof}
    Since local diffeomorphisms are open maps, $f(\Omega_1)$ is an open subset of $\Omega_2$. On the other hand, since $\Omega_1$ is compact and $\Omega_2$ is Hausdorff, $f(\Omega_1)$ is closed. By the connectedness of $\Omega_2$, $f$ is surjective. Every surjective proper local diffeomorphism is a covering map, hence $f$ is a covering map. However, since $f$ maps $\partial\Omega_1$ to $\partial\Omega_2$ as an embedding, it must be a diffeomorphism between $\Omega_1$ and $\Omega_2$.
\end{proof}

\begin{proof}[Proof of Proposition \ref{prop:iso-to-geod-ball}] 
    By Lemma \ref{lem: embed into space form}, $(\Omega,g)$ can be isometrically embedded into a space form $M_\kappa$. Let $\Pi:\mathbb M_\kappa\to M_\kappa$ be the Riemannian universal covering. Then $\Pi^{-1}(\Sigma)$ is a disjoint union $\sqcup_j \partial B_j$, where each $B_j$ is a geodesic ball of radius $\rho$. By Lemma \ref{lem: disjoint balls}, the balls $B_j$ are mutually disjoint. Since $\Pi$ is a local isometry, it preserves the direction of the mean curvature vector along $\Pi^{-1}(\Sigma)$. It follows that $\Pi$ maps a collar neighborhood of $\partial B_j$ in $B_j$ to a collar neighborhood of $\Sigma$ in $\Omega$. Therefore, $\Pi$ maps $B_j$ into $\Omega$ since $B_j$ is connected. By Lemma \ref{lem: local diffeomorphism implies diffeomorphism},  $\Omega$ is diffeomorphic to $\mathbb B_\kappa(\rho)$. In fact, it is isometric to $\mathbb B_\kappa(\rho)$ because $\Pi$ is a local isometry. 
\end{proof} 

We now prove Corollary \ref{cor: determine 3d geodesic balls in orientable space forms}.
\begin{proof}[Proof of Corollary \ref{cor: determine 3d geodesic balls in orientable space forms}]
By Lemma \ref{lem: spectrum controls chi sphere}, $\partial\Omega$ is either isometric to $\partial\mathbb B_\kappa(\rho)$ or is the disjoint union of two copies of $\partial\mathbb B_\kappa(\rho)/\pm$. The second possibility cannot occur: $\mathbb M_\kappa$ is orientable, and every boundary component of the domain $\Omega\subset\mathbb M_\kappa$ is therefore orientable. Consequently $\partial\Omega$ is connected, isometric to $\partial\mathbb B_\kappa(\rho)$, and totally umbilical with principal curvature $1/\tn_\kappa(\rho)$ with respect to the inward unit normal. By the fundamental theorem of hypersurfaces in space forms \cite[Theorem 20]{spivak1970comprehensive}, $\partial\Omega$ is a geodesic sphere of radius $\rho$ and hence, $\Omega$ is a geodesic ball of radius $\rho$. 
\end{proof}

\subsection{Proof of Proposition \ref{prop: spectrum determines boundary}}
When $\kappa>0$, Lemma \ref{lem: disjoint balls} fails to be true unless we impose extra conditions on the radii of the geodesic balls in consideration. We use a different argument to show that if $(\Omega^3,g)$ is Steklov isospectral to a geodesic ball $\mathbb B_\kappa(\rho)$, then its boundary is orientable.
\begin{lemma}\label{lem: spectrum implies orientability} 
   Suppose $\kappa > 0$. If $\Sigma$ is totally umbilical and is  isometric to the disjoint union of two copies of $\partial\mathbb B_\kappa(r)/\pm$,  then  $(\Omega,g)$ is a warped product
    \[\left(\mathbb RP^2\times \left[\rho,\frac{\pi}{\sqrt{\kappa}}-\rho\right], \sn_\kappa^2(t)g_{\mathbb RP^2} + dt^2\right),\]
    for $\rho=\min\{r, \frac{\pi}{\sqrt{\kappa}}-r\}$, where $g_{\mathbb RP^2}$ is the spherical metric on $\mathbb RP^2$ of constant sectional curvature $1$.
\end{lemma}
\begin{proof}
    Let $(\Omega', g')$ denote the orientation double covering of $(\Omega,g)$. The boundary of $\Omega'$ is totally umbilical and is isometric to the disjoint union of two copies of $\partial\mathbb B_\kappa(\rho)$. By Lemma \ref{lem: embed into space form}, $\Omega'$ can be embedded into a compact space form $M$. Let $\tau$ be the nontrivial deck transformation of $\Omega'$. The restriction of $\tau$ to each connected component of $\partial \Omega'$ is just the antipodal map. When we glue in the geodesic balls, the antipodal map extends over the ball by sending $v$ to $-v$ in their geodesic polar coordinates. Therefore, we can extend $\tau$ to an involution on $M$ that reverses orientation, with exactly two fixed points: the centers of the balls. 

    By the Killing--Hopf Theorem, $M$ is diffeomorphic to $S^3/G$ for some discrete subgroup $G$ of $SO(4)$. Let $\Pi:S^3\to M$ denote the covering map. Since $S^3$ is simply-connected, the map $\tau\circ\Pi$ has a lift $\tilde\tau: S^3\to S^3$,  unique up to composition with a deck transformation in $G$. Let $p\in M$ be a fixed point of $\tau$, and choose a lift $\tilde p\in S^3$. By composing $\tilde\tau$ with an element of $G$, we may assume that $\tilde\tau(\tilde p) = \tilde p$. Since $\tau^2 = 1$, we have $\tilde\tau^2\in G$ and $\tilde\tau^2(\tilde p) = \tilde p$. Since $G$ acts on $S^3$ freely, $\tilde\tau^2$ must be the identity map. 

    Let $\phi$ be an arbitrary orientation-reversing isometric involution on $S^3$. Equivalently, $\phi$ is an orientation-reversing orthogonal transformation on $\mathbb R^4$. The eigenvalues of $\phi$ are $\pm 1$. Since $\det(\phi) = -1$, the dimension of $+1$-eigenspace is either $1$ or $3$, which means that $\phi$ fixes either two points or an equator of $S^3$. The original involution $\tau$ has only two fixed points. Hence it must be the first case for $\tilde\tau$. Without loss of generality, we assume they are the north pole and the south pole.

    Let $\theta\in G$, then $\theta\tilde\tau$ is also an orientation-reversing isometry of $S^3$, which descends to the same involution $\tau$ on $M$. The above argument applies to $\theta\tilde\tau$, too. Therefore, $\theta\tilde\tau$ has a $1$-dimensional $+1$-eigenspace. The $-1$-eigenspace of $\tilde\tau$ and the $-1$-eigenspace of $\theta\tilde\tau$ have at least a $2$-dimensional intersection, which means the $+1$-eigenspace of $(\theta\tilde\tau)\tilde\tau$ has dimension at least $2$. However, $(\theta\tilde\tau)\tilde\tau$ = $\theta\in G$ since $\tilde\tau$ is an involution. Hence $\theta$ is the identity map since the $G$-action on $S^3$ is free. Therefore, $G = \{\Id\}$ and thus, $M \simeq S^3$.

    From the above discussion, $\Omega'$ is diffeomorphic to $S^3\setminus (B_1\cup B_2)$, where $B_1$ and $B_2$ are geodesic balls in $S^3$. Since the covering map $\Omega'\to\Omega$ is a local isometry, $\Omega'$ is isometric to
    \[\left(\mathbb S^2\times \left[\rho,\frac{\pi}{\sqrt{\kappa}}-\rho\right], \sn_\kappa^2(t)g_{\mathbb S^{2}} + dt^2\right).\]
    The lemma follows by taking the quotient of $\Omega'$.
\end{proof}

\begin{lemma}
    \label{lem: nonisopectral}
    The warped product space
    \[M:=\left(\mathbb RP^2\times \left[\rho,\frac{\pi}{\sqrt{\kappa}}-\rho\right], \sn_\kappa^2(t)g_{\mathbb RP^2} + dt^2\right),\qquad \rho<\frac{\pi}{2\sqrt{\kappa}},\]
    is not Steklov isospectral to $\mathbb B_\kappa(\rho)$.
\end{lemma}
\begin{proof}
    Recall that the Lie group $\Ogroup{3}$ acts on $\mathbb B_\kappa(\rho)$ isometrically by rotation, and there is a decomposition 
    \[L^2(\partial\mathbb B_\kappa(\rho))\simeq\widehat\bigoplus_{\ell=0}^\infty\mathcal H_\ell,\]
    where $\mathcal H_\ell$ is the eigenspace of the boundary Laplacian corresponding to the eigenvalue $\lambda_\ell=\ell(\ell+1)$, $\dim\mathcal H_\ell = 2\ell+1$. Since the Dirichlet-to-Neumann operator $\Lambda(\mathbb B_\kappa(\rho))$ commutes with the isometric action, by Schur's lemma, we have the operator decomposition
    \[\Lambda(\mathbb B_\kappa(\rho))|_{\mathcal H_\ell} = \mu_\ell \Id_{\mathcal H_\ell}.\]
    In particular, the first nonzero Steklov eigenvalue of $\mathbb B_\kappa(\rho)$ has multiplicity at least $3$.

    In fact, we can show the strict inequalities $\mu_j<\mu_k$ if $j<k$; hence the first nonzero eigenvalue $\mu_1$ has exact multiplicity $3$. Note that the Laplacian on $\mathbb B_\kappa(\rho)$ has the decomposition
    \[\Delta = \partial_r^2 + \frac{2}{\tn_\kappa(r)}\partial_r + \frac{1}{\sn_\kappa^2(r)}\Delta_{\mathbb S^2},\]
    where $r$ is the radial distance and $\Delta_{\mathbb S^2}$ is the Laplacian on the standard unit sphere. Take $\psi\in\mathcal H_\ell$. By separation of variables, the unique harmonic extension of $\psi$ is given by $\mathcal H\psi(r,\theta) = y_\ell(r)\psi(\theta)$ in the geodesic polar coordinates at the center, where $y_\ell(r)$ satisfies the ordinary differential equation:
    \[\bigl(\sn_\kappa^2(r)y_\ell'(r)\bigr)'
=\ell(\ell+1)y_\ell(r),\qquad y_\ell(\rho) = 1.\]
    By the Fuchs-Frobenius theory \cite[Theorem 4.1]{olver1997asymptotics}, such a solution exists and is unique. Moreover, this solution is analytic and satisfies $y_\ell(r) = O(r^\ell)$ as $r\to 0$.
         
    Suppose that $y_\ell(r_0) = 0$ for some $r_0\in(0,\rho]$. Multiplying the equation by $y_\ell$ and integrating up to $r_0$ gives
    \[-\int_0^{r_0}\sn_\kappa^2(r)\left(y_\ell'(r)\right)^2dr = \ell(\ell+1)\int_0^{r_0}y_l^2dr\ge 0\]
    which forces $y_\ell\equiv0$. Therefore, $y_\ell > 0$ on $(0,\rho]$.
    
    For $j<k$, let $W = y_k'y_j-y_j'y_k$ denote the Wronskian. It satisfies
\[
\frac{d}{dr}\left[
\sn_\kappa^2(r)W(r)
\right]
=\bigl(k(k+1)-j(j+1)\bigr)y_j(r)y_k(r).
\]
Since $W(r) = O(r^{j+k-1})$, integrating from $0$ to $\rho$ and using
$y_j(\rho)=y_k(\rho)=1$ yield
\[
\sn_\kappa^2(\rho)(\mu_k-\mu_j)
=\bigl(k(k+1)-j(j+1)\bigr)
\int_0^\rho y_j(r)y_k(r)\,dr>0.
\]
Thus the first positive Steklov eigenvalue of the ball is $\mu_1$, with multiplicity $3$.

    Similarly, the Lie group $\operatorname{SO}_3(\mathbb R)\times \mathbb Z_2$ acts on $M$ isometrically by setting
    \[(A,\tau).(p,t) = (A.p, \pi/\sqrt{\kappa}-t),\]
    where $\tau$ is the nontrivial element in $\mathbb Z_2$. 
    Since 
    \[L^2(\mathbb RP^2)\simeq\widehat\bigoplus_{\ell = 0}^\infty \mathcal H_{2\ell},\]
    we have the decomposition
    \[L^2(\partial M)\simeq\widehat\bigoplus_{\ell = 0}^\infty \mathcal V_{2\ell}^+\oplus\mathcal V_{2\ell}^-,\]
    where
    \[\mathcal V_{2\ell}^\pm = \{(f,\pm f)\mid f\in\mathcal H_{2\ell}\}\subset\mathcal H_{2\ell}\oplus \mathcal H_{2\ell},\quad\dim\mathcal V_{2\ell}^\pm = 4\ell + 1,\]
    since $\partial M$ consists of two copies of $\partial\mathbb B_\kappa(\rho)/\pm\simeq\mathbb RP^2$. By Schur's lemma again, we have the operator decomposition
    \[\Lambda(M)|_{\mathcal V_{2\ell}^\pm} = \mu_{2\ell}^\pm \Id_{\mathcal V_{2\ell}^\pm}.\]
    By an argument similar to the one above, we can show that $\{\mu_{2\ell}^+\}_{\ell = 0}^\infty$ and $\{\mu_{2\ell}^-\}_{\ell = 0}^\infty$ are strictly increasing sequences.
    Therefore, since $\mu_0^+$ corresponds to the zero eigenvalue, the first nonzero Steklov eigenvalue is $\min\{\mu_2^+,\mu_0^-\}$, whose multiplicity is $1$ if $\mu_2^+ > \mu_0^-$, $5$ if $\mu_2^+ < \mu_0^-$, or $6$ if $\mu_2^+ = \mu_0^-$. However, for the ball, $\mu_1$ has multiplicity $3$. This proves the proposition.
\end{proof}

\begin{proof}[Proof of Proposition \ref{prop: spectrum determines boundary}]

    By Lemma \ref{lem: spectrum controls chi sphere}, $\Sigma$ is totally umbilical with principal curvatures equal to $1/\tn_\kappa(\rho)$ and $\Sigma$ is isometric to either $\partial\mathbb B_\kappa(\rho)$ or the disjoint union of two copies of $\partial\mathbb B_\kappa(\rho)/\pm$. To prove the desired proposition, it suffices to show that $\Sigma$ is orientable. If $\kappa > 0$, this is a consequence of Lemma \ref{lem: spectrum implies orientability} and Lemma \ref{lem: nonisopectral}.

    Now suppose that $\kappa \le 0$ and $\Sigma$ is non-orientable. Let $\Omega'$ be the Riemannian orientation double covering of $\Omega$, whose boundary is totally umbilical and is isometric to the disjoint union of two copies of $\partial\mathbb B_\kappa(\rho)$. The same argument as in the proof of Proposition \ref{prop:iso-to-geod-ball} applies to $\Omega'$. In particular, this means that $\partial\Omega'$ has only one connected component, which is a contradiction. 
\end{proof}

\appendix

\section{Proof of Proposition \ref{prop: faithful} and Proposition \ref{prop: evaluation invariance is algebraic invariance}}
\label{apd: invariance}

\begin{lemma}
    \label{lem: R is integral}
    The algebra $\mathcal R = \mathbb R[\mathcal G][\mathsf G]/(\mathsf G^2-\det\mathsf g)$ is an integral domain.
\end{lemma}
\begin{proof}
    Denote $\mathcal A = \mathbb R[\mathcal G]$. Define an endomorphism $\psi: \mathcal A\to \mathcal A$ by setting
    \[\psi(\mathsf g_{\alpha\beta,c_1\dots c_w}) = \begin{cases}
        \mathsf g_{\alpha\alpha} & \text{if $\alpha = \beta$ and $w=0$},\\
        0 & \text{otherwise.}
    \end{cases}\]
    If there exists $\mathsf f\in\mathcal A$ such that $\mathsf f^2 = \det\mathsf g$, then 
    \[\psi(\mathsf f)^2 = \psi(\det\mathsf g) = \mathsf g_{11}\dots\mathsf g_{n-1,n-1},\] 
    which is impossible since $\mathcal A$ is a unique factorization domain and the degree of an irreducible factor in a square is always even. Therefore, the polynomial $\mathsf G^2 - \det(\mathsf g)$ is irreducible in $\mathcal A[\mathsf G]$, which implies that $\mathcal R$ is an integral domain.
\end{proof}
\begin{proof}[Proof of Proposition \ref{prop: faithful}]
    First we show that if $\mathsf p$ is a nonzero polynomial in the variables $\mathcal G$, and if $\mathsf p(\phi,g) = 0$ for any Riemannian manifold $(\Omega,g)$ with an $\varepsilon$-collar and any boundary normal coordinate system $\phi$, then $\mathsf p = 0$. There exists $a_{\alpha\beta,c_1\dots c_w}\in\mathbb R$ for any $1\le \alpha,\beta\le n-1,$ $1\le c_1,\dots,c_w\le n,w\in\mathbb N_0$ such that the matrix $(a_{\alpha\beta})$ is positive definite and the evaluation of $\mathsf p$ at $\mathsf g_{\alpha\beta,c_1\dots c_w}= a_{\alpha\beta,c_1\dots c_w}$
    is nonzero, as $\mathsf p$ is a nonzero polynomial. By Borel's lemma, there exists a metric $g$ on $\mathbb R^{n-1}\times [0,\varepsilon)$ such that
    \[g_{\alpha\beta,c_1\dots c_w}(0,\dots,0) = a_{\alpha\beta,c_1\dots c_w}.\]
    It then follows that $\mathsf p(\phi,g)$ is nonzero at the origin, which is a contradiction.

    Next we suppose that $\mathsf p\in\mathcal P$ satisfies the stated assumption. Multiplying by a sufficiently large power of $\mathsf G$, we may assume that $\mathsf p\in\mathcal R$. We can write $\mathsf p = \mathsf u + \mathsf G\mathsf v$, where $\mathsf u,\mathsf v\in\mathbb R[\mathcal G]$. If we set $\mathsf q := (\mathsf u + \mathsf G\mathsf v)(\mathsf u - \mathsf G\mathsf v) = \mathsf u^2 - \mathsf G^2\mathsf v^2$, then $\mathsf q$ is a polynomial in the variables $\mathcal G$ and satisfies $\mathsf q(\phi,g) = 0$ for any Riemannian manifold with an $\varepsilon$-collar, and any boundary normal coordinates $\phi$. By the above discussion, $\mathsf q = 0$. However, it follows from Lemma \ref{lem: R is integral} and the assumption $\mathsf u + \mathsf G\mathsf v\neq0$ that $\mathsf u - \mathsf G\mathsf v$ = 0. As a result, the evaluations of $\mathsf u = \frac{\mathsf u + \mathsf G\mathsf v}{2} + \frac{\mathsf u - \mathsf G\mathsf v}{2}$ and $\mathsf v = \frac{\mathsf u + \mathsf G\mathsf v}{2\mathsf G} + \frac{\mathsf u - \mathsf G\mathsf v}{2\mathsf G}$ are always zero. Therefore, $\mathsf u = \mathsf v = 0$ by the above discussion again. Hence, $\mathsf p = 0$.
\end{proof}

To prove Proposition \ref{prop: evaluation invariance is algebraic invariance}, we introduce the following notation. Let $(\Omega,g)$ be a Riemannian manifold with smooth boundary $\Sigma$ and an $\varepsilon$-collar. The $\varepsilon$-collar neighborhood is diffeomorphic to $\Sigma\times[0,\varepsilon)$. Under this identification, the Riemannian metric $g$ takes the form
\[g = g_{\alpha\beta}(q,x^n)dx^\alpha\otimes dx^\beta + dx^n\otimes dx^n.\]
Let $\Sigma_t$ denote the Riemannian submanifold $\{x^n = t\}$ of $\Sigma\times[0,\varepsilon)$, and let $F(\Omega)$ be the principal bundle over $\Sigma\times[0,\varepsilon)$, whose restriction to $\Sigma_t$ is the orthonormal frame bundle of $\Sigma_t$. Given $\mathrm e \in F(\Omega)_{(q,t)}$, there is a unique geodesic normal coordinate system on $\Sigma_t$ associated to $\mathrm e$, denoted by  $\phi_{\mathrm e}$. Note that as a smooth manifold $\Sigma_t$ is diffeomorphic to $\Sigma$. Our definition of the invariance for $\mathsf p\in\mathcal P$ now translates into: for any Riemannian manifold with boundary $\Omega$ as above and any $(q,t)\in\Sigma\times[0,\varepsilon)$, the value $\mathsf p(\phi_\mathrm e,g)(q,t)$ is independent of $e\in F(\Omega)_{(q,t)}$.

\begin{proof}[Proof of Proposition \ref{prop: evaluation invariance is algebraic invariance}]
    Suppose that $\mathsf p\in\mathcal P$ satisfies $\mathsf p - A.\mathsf p\in\mathcal J$ for all $A\in\Ogroup{n-1}$. Given a Riemannian manifold $(\Omega,g)$ with an $\varepsilon$-collar, for any $\mathrm e\in F(\Omega)_{(q,t)}$, one has $\mathsf p(\phi_{\mathrm e.A},g)(q,t) = A.\mathsf p(\phi_\mathrm e,g)(q,t) = \mathsf p(\phi_\mathrm e,g)(q,t)$ since $\mathsf r(\phi_\mathrm e,g)(q,t) = 0$ for any $\mathsf r\in\mathcal J$.
    
    Conversely, let $\mathsf p\in\mathcal P$ be invariant and $A\in\Ogroup{n-1}$. We show that $\mathsf q:= \mathsf p-A.\mathsf p\in\mathcal J$. Note that $\mathsf q$ is also invariant because $A.\mathsf p(\phi_\mathrm e,g)(q,t) = \mathsf p(\phi_{\mathrm e.A},g)(q,t)$, for any $\mathrm e\in F(\Omega)_{(q,t)}$. Moreover, $\mathsf q(g) = 0$ for every Riemannian manifold $(\Omega,g)$ with an $\varepsilon$-collar. By the definition, we have
    \begin{align*}
        \mathcal O = \mathcal P/\mathcal J & \simeq \left(\frac{\mathcal R}{(\mathsf G-1,\mathsf g_{\alpha\beta}-\delta_{\alpha\beta},\mathsf g_{\alpha\beta,\gamma})}\right)_\mathsf G\\
    	& \simeq\left(\frac{\mathbb R[\mathcal G,\mathsf G]}{(\mathsf G-1,\mathsf g_{\alpha\beta}-\delta_{\alpha\beta},\mathsf g_{\alpha\beta,\gamma})}\right)_\mathsf G\\
    	& \simeq\mathbb R[\mathcal G'],
    \end{align*} 
    where $\mathcal G' = \mathcal G\setminus\{\mathsf g_{\alpha\beta},\mathsf g_{\alpha\beta,\gamma}\mid 1\le\alpha,\beta,\gamma\le n-1\}$. Therefore, we may view the image $\bar{\mathsf q}$ of $\mathsf q$ in $\mathcal O$ as a polynomial in the variables in $\mathcal G'$. Then a similar argument as in the proof of Proposition \ref{prop: faithful} shows that $\bar{\mathsf q} = 0$.
\end{proof}

\begin{proof}[Proof of Lemma \ref{lem: independent of phi}]
    Since $\mathsf p(\phi,g)(q,t)$ is independent of the choice of $\phi$, it is always zero since $\mathsf p\in\mathcal J$ and we can choose $\phi$ to be $\phi_\mathrm e$ for some $\mathrm e\in F(\Omega)_{(q,t)}$. Therefore $\mathsf p = 0$ by Proposition \ref{prop: faithful}.
\end{proof}

\section{Computation of derivatives}
\label{apd: computation}

\begin{proof}[Proof of Formulas \eqref{eq: g (2,0)}--\eqref{eq: g (2,1)}]
    Formula \eqref{eq: g (2,0)} follows from the Taylor expansion of the metric tensor at the origin of a geodesic normal coordinate system, that is, 
    \[g_{\alpha\beta}(x) = \delta_{\alpha\beta} + \frac{1}{3}\Rm_{\alpha\gamma\delta\beta}^\top x^\delta x^\gamma + O(|x|^3).\] 

    We prove the remaining identities.  
    We start with the congruences
    \[
        \mathsf g_{\alpha\beta}\equiv\delta_{\alpha\beta},
        \qquad \mathsf g_{\alpha\beta,\gamma}\equiv0\mod\mathcal J.
    \]
    Since $\mathsf g_{nn}=1$ and $\mathsf g_{\alpha n}=0$, the only Christoffel symbols involving the normal direction that are needed below are
    \[
        \mathsf \Gamma_{\alpha\beta}^{n} = \mathsf h_{\alpha\beta},
        \qquad
        \mathsf \Gamma_{\alpha n}^{\beta} = -\mathsf g^{\beta\gamma}\mathsf h_{\alpha\gamma},
        \qquad
        \mathsf \Gamma_{nn}^{a} = \mathsf \Gamma_{\alpha n}^{n}=0.
    \]

    First, by the definition of the curvature tensor adopted above,
    \begin{align*}
        \mathsf{Rm}^\bot_{\alpha\beta}
        =\mathsf{Rm}_{\alpha n\beta n}
        &=\partial_n\mathsf\Gamma_{\alpha\beta}^{n}
          -\mathsf\Gamma_{n\beta}^{\gamma}\mathsf\Gamma_{\alpha\gamma}^{n}\\
        &=\partial_n\mathsf h_{\alpha\beta}
          +\mathsf g^{\gamma\delta}\mathsf h_{\alpha\gamma}\mathsf h_{\beta\delta}.
    \end{align*}
    Because $\mathsf h_{\alpha\beta}=-\frac12 \mathsf g_{\alpha\beta,n}$, this is equivalent to
    \[
        \mathsf g_{\alpha\beta,nn}
        =2\bigl(\mathsf g^{\gamma\delta}\mathsf h_{\alpha\gamma}\mathsf h_{\beta\delta}
        -\mathsf{Rm}^\bot_{\alpha\beta}\bigr),
    \]
    which proves \eqref{eq: g (0,2)}.

    To differentiate this identity once more, put
    \[
        \mathsf S_{\alpha\beta}:=\mathsf g^{\gamma\delta}\mathsf h_{\alpha\gamma}\mathsf h_{\beta\delta}.
    \]
    The identities
    \[
        \partial_n \mathsf g^{\gamma\delta}=2\mathsf h^{\gamma\delta},
        \qquad
        \partial_n \mathsf h_{\alpha\beta}=\mathsf{Rm}^\bot_{\alpha\beta}-\mathsf S_{\alpha\beta}
    \]
    imply, after cancellation of the cubic terms in $h$, that
    \begin{equation}\label{eq: normal derivative h square}
        \partial_n \mathsf S_{\alpha\beta}
        =\mathsf g^{\gamma\delta}\bigl(
        \mathsf h_{\alpha\gamma}\mathsf {Rm}^\bot_{\beta\delta}
        +\mathsf h_{\beta\gamma}\mathsf{Rm}^\bot_{\alpha\delta}\bigr).
    \end{equation}
    Moreover, since $\mathsf \Gamma_{n\alpha}^{\gamma}=-\mathsf g^{\gamma\delta}\mathsf h_{\alpha\delta}$ and $\mathsf \Gamma_{nn}^{a}=0$, the covariant normal derivative of $\mathsf {Rm}^\bot$ satisfies
    \begin{equation}\label{eq: partial versus covariant normal Rm bot}
        \mathsf {Rm}^\bot_{\alpha\beta;n}
        =\partial_n\mathsf{Rm}^\bot_{\alpha\beta}
        +\mathsf g^{\gamma\delta}\mathsf h_{\alpha\gamma}\mathsf {Rm}^\bot_{\beta\delta}
        +\mathsf g^{\gamma\delta}\mathsf h_{\beta\gamma}\mathsf{Rm}^\bot_{\alpha\delta}.
    \end{equation}
    Differentiating \eqref{eq: g (0,2)} and using
    \eqref{eq: normal derivative h square}--\eqref{eq: partial versus covariant normal Rm bot} yields
    \[
        \mathsf g_{\alpha\beta,nnn}
        =4\mathsf g^{\gamma\delta}\bigl(
        \mathsf h_{\alpha\gamma}\mathsf {Rm}^\bot_{\beta\delta}
        +\mathsf h_{\beta\gamma}\mathsf {Rm}^\bot_{\alpha\delta}\bigr)
        -2\mathsf {Rm}^\bot_{\alpha\beta;n},
    \]
    proving \eqref{eq: g (0,3)}.

    It remains to prove the mixed tangential--normal identity.  Differentiating the coordinate formula for the tangential Christoffel symbols and using \eqref{eq: g (2,0)}, together with the usual symmetries of $\mathsf{Rm}^\top$, gives
    \begin{align}
        \partial_\delta(\mathsf \Gamma^\top)_{\gamma\alpha}^{\epsilon}
        &\equiv \frac12\bigl(
        \mathsf g_{\epsilon\alpha,\gamma\delta}
        +\mathsf g_{\gamma\epsilon,\alpha\delta}
        -\mathsf g_{\gamma\alpha,\epsilon\delta}\bigr)\notag\\
        &\equiv \frac13\bigl(
        \mathsf {Rm}^\top_{\gamma\delta\alpha\epsilon}
        +\mathsf {Rm}^\top_{\alpha\delta\gamma\epsilon}\bigr)
        \equiv-\mathsf g_{\gamma\alpha,\epsilon\delta}\mod\mathcal J.
        \label{eq: derivative tangential Christoffel}
    \end{align}
    On the other hand, the definition of the second tangential covariant derivative of the $(0,2)$-tensor $h$ gives
    \begin{align*}
        \mathsf h_{\alpha\beta:\gamma\delta}
        &\equiv \partial_\delta\partial_\gamma \mathsf h_{\alpha\beta}
        -\partial_\delta(\mathsf \Gamma^\top)_{\gamma\alpha}^{\epsilon}\mathsf h_{\epsilon\beta}
        -\partial_\delta(\mathsf \Gamma^\top)_{\gamma\beta}^{\epsilon}\mathsf h_{\alpha\epsilon}.
    \end{align*}
    Substituting \eqref{eq: derivative tangential Christoffel} and using
    $\mathsf g_{\alpha\beta,n}=-2\mathsf h_{\alpha\beta}$, we obtain
    \begin{align*}
        \mathsf g_{\alpha\beta,\gamma\delta n}
        &\equiv-2\mathsf h_{\alpha\beta:\gamma\delta}
        +2\mathsf h_{\beta\epsilon}\mathsf g_{\gamma\alpha,\epsilon\delta}
        +2\mathsf h_{\alpha\epsilon}\mathsf g_{\gamma\beta,\epsilon\delta},
    \end{align*}
    which is \eqref{eq: g (2,1)}.  
\end{proof}

In Section \ref{sec: proof of thm}, we used formulas for the derivatives of some specific metric polynomials without proof; we prove them here. It follows from the Gauss equation that 
\[\mathsf R^\top - \mathsf R = 2(\mathsf H_2 - \mathsf{Ric}^\bot).\]
Thus it suffices to compute $\partial_n\mathsf H_1$, $\partial_n\mathsf H_2$, $\partial_n\mathsf\Delta^\top\mathsf H_1$ and $\partial_n\langle\mathsf h, \mathsf{Rm}^\bot\rangle$, as the others can be derived from these using the Leibniz rule. Recall that by definition, $\mathsf h_{\alpha\beta} = -\frac{1}{2}\mathsf g_{\alpha\beta,n}$, $\mathsf H_1 = \frac{1}{2}\mathsf g^{\alpha\beta}\mathsf h_{\alpha\beta}$, $\mathsf H_2 = \mathsf G^{-2}\det\mathsf(h_{\alpha\beta})$, where $\det\mathsf(h_{\alpha\beta}) = \mathsf h_{11}\mathsf h_{22} - \mathsf h_{12}^2$, $\mathsf{\Delta^\top H_1} = \mathsf g^{\alpha\beta}(\mathrm{Hess}^\top \mathsf H_1)_{\alpha\beta}$, $\langle\mathsf h, \mathsf{Rm}^\bot\rangle = \mathsf g^{\alpha\gamma}\mathsf g^{\beta\delta}\mathsf h_{\alpha\beta}\mathsf{Rm}_{n\gamma n\delta}$. In what follows, all repeated indices are summed over.

\begin{lemma}\label{lem: partial n}
    We have
    \begin{enumerate}
        \item $\partial_n\mathsf H_1 = 2\mathsf H_1^2 - \mathsf H_2 + \frac{1}{2}\mathsf{Ric}^\bot$;
        \item $\partial_n\mathsf H_2 = 2\mathsf H_1\mathsf H_2 + 2\mathsf H_1\mathsf{Ric}^\bot - \langle \mathsf h, \mathsf{Rm}^\bot\rangle$;
        \item $\partial_n\langle\mathsf h, \mathsf{Rm}^\bot\rangle = \langle\mathsf h^2,\mathsf{Rm}^\bot\rangle + \langle\mathsf{Rm}^\bot,\mathsf{Rm}^\bot\rangle + \langle\mathsf h, \nabla_n\mathsf{Rm}^\bot\rangle$;
        \item $\partial_n\mathsf{\Delta^\top H_1} = 2\langle \mathsf h, \mathrm{Hess}^\top\mathsf H_1\rangle + 4\mathsf H_1\mathsf{\Delta^\top H_1} + 2\langle\mathsf{\nabla^\top H_1},\mathsf{\nabla^\top H_1}\rangle + 2\langle\dvg^\top\mathsf h, \nabla^\top\mathsf H_1\rangle - \mathsf{\Delta^\top H_2} + \frac{1}{2}\mathsf{\Delta^\top \Ric^\bot}$.
    \end{enumerate}
\end{lemma}

\begin{proof}
    Since $\mathsf g^{\alpha\beta}\mathsf g_{\beta\gamma} = \delta_{\alpha\gamma}$, we have $\partial_c\mathsf g^{\alpha\beta} \equiv -\mathsf g_{\alpha\beta,c}$ for any $c$. Using \eqref{eq: g (0,2)}, we see
    \begin{align*}
        \partial_n\mathsf H_1
        & = \frac{1}{2}\partial_n\mathsf g^{\alpha\beta}\mathsf h_{\alpha\beta} + \frac{1}{2}\mathsf g^{\alpha\beta}\partial_n\mathsf h_{\alpha\beta}\\
        & \equiv -\frac{1}{2}\mathsf g_{\alpha\beta,n}\mathsf h_{\alpha\beta} - \frac{1}{2}\mathsf g^{\alpha\beta}(\mathsf g^{\gamma\delta}\mathsf h_{\alpha\gamma}\mathsf h_{\beta\delta} - \mathsf{Ric}_{\alpha\beta}^\bot)\\
        & \equiv \mathsf h_{\alpha\beta}\mathsf h_{\alpha\beta} - \frac{1}{2}\mathsf h_{\beta\gamma}\mathsf h_{\beta\gamma} +  \frac{1}{2}\mathsf{Ric}^\bot\\
        & \equiv \frac{1}{2}\langle\mathsf h,\mathsf h\rangle + \frac{1}{2}\mathsf{Ric}^\bot.
    \end{align*}
    Note that 
    \[4\mathsf H_1^2-2\mathsf H_2\equiv(\mathsf h_{11} + \mathsf h_{22})^2 - 2(\mathsf h_{11}\mathsf h_{22} - \mathsf h_{12}^2) = \mathsf h_{11}^2 + \mathsf h_{22}^2 + 2\mathsf h_{12}^2\equiv\langle\mathsf h,\mathsf h\rangle.\]
    Then (1) follows from Lemma \ref{lem: independent of phi}. It also follows that
    \[\partial_n(\mathsf h_{11}\mathsf h_{22} - \mathsf h_{12}^2) \equiv - 2\mathsf H_1\mathsf H_2 + 2\mathsf H_1\mathsf{Ric}^\bot - \langle\mathsf h,\mathsf{Rm}^\bot\rangle.\]
    Since
    \[\partial_n\mathsf G \equiv \frac{1}{2}(\mathsf g_{11,n} + \mathsf g_{22,n}) \equiv - 2\mathsf H_1,\]
    we have
    \begin{align*}
        \partial_n\mathsf H_2
        & = \partial_n(\mathsf G^{ - 2}\det(\mathsf h_{\alpha\beta}))\\
        & \equiv - 2\partial_n\mathsf G\det(\mathsf h_{\alpha\beta}) + \partial_n\det(\mathsf h_{\alpha\beta})\\
        & \equiv 2\mathsf H_1\mathsf H_2 + 2\mathsf H_1\mathsf{Ric}^\bot - \langle\mathsf h, \mathsf{Rm}^\bot\rangle.
    \end{align*}
    By Lemma \ref{lem: independent of phi}, this shows (2).

    Note that for any $\alpha$, $\beta$, the Christoffel symbol $\mathsf\Gamma_{\alpha n}^\beta$ can be computed as
    \[\mathsf\Gamma_{\alpha n}^\beta = \frac{1}{2}\mathsf g^{\beta\gamma}(\mathsf g_{\gamma n,\alpha} + \mathsf g_{\alpha\gamma,n} - \mathsf g_{\alpha n,\gamma})\equiv \frac{1}{2}\mathsf g_{\alpha\beta,n}\equiv-\mathsf h_{\alpha\beta}.\]
    Therefore, 
    \begin{align*}
        \partial_n\langle\mathsf h, \mathsf{Rm}^\bot\rangle
        & = \partial_n(\mathsf g^{\alpha\gamma}\mathsf g^{\beta\delta}\mathsf h_{\alpha\beta}\mathsf{Rm}_{n\gamma n\delta})\\
        & = \partial_n\mathsf g^{\alpha\gamma}\mathsf g^{\beta\delta}\mathsf h_{\alpha\beta}\mathsf{Rm}_{n\gamma n\delta} + \mathsf g^{\alpha\gamma}\partial_n\mathsf g^{\beta\delta}\mathsf h_{\alpha\beta}\mathsf{Rm}_{n\gamma n\delta} + \mathsf g^{\alpha\gamma}\mathsf g^{\beta\delta}\partial_n\mathsf h_{\alpha\beta}\mathsf{Rm}_{n\gamma n\delta}\\
        &\qquad  + \mathsf g^{\alpha\gamma}\mathsf g^{\beta\delta}\mathsf h_{\alpha\beta}\partial_n\mathsf{Rm}_{n\gamma n\delta}\\
        &\equiv 2\mathsf h_{\alpha\gamma}\mathsf h_{\alpha\beta}\mathsf{Rm}_{n\gamma n\beta} + 2\mathsf h_{\beta\delta}\mathsf h_{\alpha\beta}\mathsf{Rm}_{n\beta n\delta} - (\mathsf h_{\alpha\gamma}\mathsf h_{\beta\gamma} - \mathsf{Rm}_{n\alpha n\beta})\mathsf{Rm}_{n\alpha n\beta}\\
        &\qquad  + \mathsf h_{\alpha\beta}(\mathsf{Rm}_{n\alpha n\beta;n} + \mathsf\Gamma_{\alpha n}^{\sigma}\mathsf{Rm}_{n\sigma n\beta} + \mathsf\Gamma_{\beta n}^{\sigma}\mathsf{Rm}_{n\alpha n\sigma})\\
        & \equiv \langle\mathsf h^2,\mathsf{Rm}^\bot\rangle + \langle\mathsf{Rm}^\bot,\mathsf{Rm}^\bot\rangle + \langle\mathsf h, \nabla_n\mathsf{Rm}^\bot\rangle.
    \end{align*}
    By Lemma \ref{lem: independent of phi}, this proves (3).

    Using the expression of the Laplacian in local coordinates, we have
    \[\mathsf{\Delta^\top H_1} = \mathsf G^{ - 1}\partial_\alpha\mathsf G\mathsf g^{\alpha\beta}\partial_\beta\mathsf H_1 + \partial_\alpha\mathsf g^{\alpha\beta}\partial_\beta\mathsf H_1 + \mathsf g^{\alpha\beta} \partial_\alpha\partial_\beta\mathsf H_1 := \mathsf p_1 + \mathsf p_2 + \mathsf p_3.\]
    
    For $\mathsf p_1$, it follows that
    \begin{align*}
        \partial_n\mathsf p_1 
        & = \partial_n\partial_\alpha\mathsf G(\mathsf G^{-1}\mathsf g^{\alpha\beta}\partial_\beta\mathsf H_1) + \partial_\alpha\mathsf G\partial_n(\mathsf G^{-1}\mathsf g^{\alpha\beta}\partial_\beta\mathsf H_1)\\
        & \equiv \partial_\alpha\partial_n\mathsf G\partial_\alpha\mathsf H_1\\
        & = -2\nabla^\top_\alpha\mathsf H_1\nabla^\top_\alpha\mathsf H_1\\
        & \equiv -2\langle\mathsf{\nabla^\top H_1},\mathsf{\nabla^\top H_1}\rangle,
    \end{align*}
    since $\partial_\gamma\mathsf p \equiv \nabla_{\gamma}\mathsf p$ for any $\mathsf p\in\mathcal P$.

    For $\mathsf p_2$, it follows that
    \begin{align*}
        \partial_n\mathsf p_2 
        & = \partial_n\partial_\alpha\mathsf g^{\alpha\beta}\partial_\beta\mathsf H_1 + \partial_\alpha\mathsf g^{\alpha\beta}\partial_n\partial_\beta\mathsf H_1\\
        & \equiv \partial_\alpha\partial_n\mathsf g^{\alpha\beta}\partial_\beta\mathsf H_1\\
        & \equiv 2\mathsf h_{\alpha\beta:\alpha}\nabla^\top_\beta\mathsf H_1\\
        & \equiv 2\langle\dvg^\top\mathsf h, \nabla^\top\mathsf H_1\rangle. 
    \end{align*}

    For $\mathsf p_3$, since $\partial_\alpha\partial_\beta\mathsf p\equiv (\mathrm{Hess}^\top\mathsf p)_{\alpha\beta}$ for any $\mathsf p\in\mathcal P$, it follows that
    \begin{align*}
        \partial_n\mathsf p_3 
        & = \partial_n\mathsf g^{\alpha\beta}\partial_\alpha\partial_\beta\mathsf H_1 + \mathsf g^{\alpha\beta}\partial_n\partial_\alpha\partial_\beta\mathsf H_1\\
        & \equiv 2\mathsf h_{\alpha\beta}\partial_\alpha\partial_\beta\mathsf H_1 + \partial_\alpha\partial_\alpha\partial_n\mathsf H_1\\
        & \equiv 2\langle \mathsf h, \mathrm{Hess}^\top\mathsf H_1\rangle + \partial_\alpha\partial_\alpha(2\mathsf H_1^2 - \mathsf H_2 + \frac{1}{2}\mathsf{Ric}^\bot)\\
        & = 2\langle \mathsf h, \mathrm{Hess}^\top\mathsf H_1\rangle + 4\mathsf H_1\mathsf{\Delta^\top H_1} + 4\langle\mathsf{\nabla^\top H_1},\mathsf{\nabla^\top H_1}\rangle - \mathsf{\Delta^\top H_2} + \frac{1}{2}\mathsf{\Delta^\top \Ric^\bot}.
    \end{align*}
    Therefore, using Lemma \ref{lem: independent of phi} again, we obtain (4). This completes the proof.
\end{proof}

\bibliographystyle{alpha}

\end{document}